\documentclass[a4paper,reqno,11pt]{amsart}
\usepackage{amssymb,amsmath,amsfonts,amsthm,stmaryrd}
\usepackage{graphicx}
\usepackage{color}
\usepackage[left=1 in, right=1 in,top=1 in, bottom=1 in]{geometry}
\makeatletter
\@namedef{subjclassname@2020}{\textup{2020} Mathematics Subject Classification}
\makeatother
\newtheorem{theorem}{Theorem}[section]
\newtheorem{lemma}[theorem]{Lemma}

\numberwithin{equation}{section}
\allowdisplaybreaks

\providecommand{\norm}[1]{\left\Vert#1\right\Vert}

\newcommand{\beq}{\begin{equation}}
\newcommand{\eeq}{\end{equation}}
\def\p{\partial}
\DeclareMathOperator{\curl}{curl}
\DeclareMathOperator{\dive}{div}
\def\dis{\displaystyle}
\def\ls{\lesssim}
\def\eps{\varepsilon}

\title[Stability of Inflow Problem for Hyperbolic Systems]{Stability of Inflow Problem for Hyperbolic Systems}

\author{Yan Guo}
\address{Division of Applied Mathematics, Brown University, Providence, RI 02912, USA}
\email[Y. Guo]{yan$\_$guo@brown.edu}

\author{Yanjin Wang}
\address{School of Mathematical Sciences\\
Xiamen University\\
Xiamen, Fujian 361005, China}
\email[Y. J. Wang]{yanjin$\_$wang@xmu.edu.cn}

\begin{document}

\begin{abstract}
Inflow BC plays a critical role in the study of hyperbolic PDE in a bounded domain. We establish $W^{1,\infty}$ stability for 1D hyperbolic conservation laws with inflow data in a bounded interval, and $W^{2,3+}$ stability of a large class of shear flows for the 3D incompressible Euler system with inflow BC in finite square or circular pipes.
\end{abstract}

 \maketitle

\section{Introduction}

Consider the scalar basic transport equation in 1D  with constant non-zero speed $\lambda\neq0$:
\[\partial_{t}f+\lambda \partial_{x}f=0.\]
The exact solution $e^{-x\frac{\lambda}{|\lambda|}+|\lambda|t}=e^{-x\text{sgn}\lambda+|\lambda|t}$ leads to
\[
\lbrack\partial_{t}+\lambda\partial_{x}]\{e^{-x\frac{\lambda}{|\lambda|}+|\lambda|t}f\}=0\text{ or }e^{-x\frac{\lambda}{|\lambda|}+|\lambda|t}f\equiv\text{constant }%
\]
along the characteristic line, in terms of the initial and boundary data. Any
non-characteristic speed creates exponential decay $e^{-|\lambda|t}$ with the
exponential spatial weight $e^{-x\frac{\lambda}{|\lambda|}}$, and $\lambda
\neq0$ is exactly needed for imposing inflow BC.   Such an exponential spatial
weight is well-known to change the spectrum, but it is rarely useful to
control nonlinearity for the Cauchy problem since $e^{-x\frac{\lambda}{|\lambda|}}$ is neither
bounded from above nor below for $x\in\mathbb{R}$. Our key observation is that
such a spatial weight $e^{-x\frac{\lambda}{|\lambda|}}\backsim1$ for bounded $x$, and hence introduces interior dissipation which is compatible with the basic inflow BC for $f$. On the other hand, the $x$-periodic BC for $f$ is not
compatible with $e^{-x\frac{\lambda}{|\lambda|}}$: $e^{-x\frac{\lambda
}{|\lambda|}}f$ destroys periodic structure of $f$ and no dissipative
effect could be obtained.

Without seeking the most comprehensive theory in this paper, we demonstrate the importance of such a dissipative mechanism for inflow BC by establishing the global stability and regularity for both 1D hyperbolic conservation laws in an interval and 3D incompressible Euler system around a large class of shear flows in either square or circular pipes.

\subsection{1D Hyperbolic Conservation Laws}

We study a general system of conservation laws in a normalized interval $I=(-1,1)\subseteq \mathbb{R}$:
\begin{equation}\label{conser}
\partial_{t}U+\partial_{x}(F(U))=0, %
\end{equation}
where the vectorial unknown $U=(U_{i}(t,x))_{i=1}^{n}$: $\mathbb{R}_{+}
\times I \rightarrow\mathbb{R}^{n}$ and $F=(F_i)_{i=1}^{n}: \mathbb{R}^{n}\rightarrow\mathbb{R}^{n}$ is a given smooth vector-valued flux function. By the chain rule, \eqref{conser} can be rewritten as
\begin{equation}\label{hyper}
\partial_{t}U+A(U)\partial_{x}U=0  %
\end{equation}
with the coefficient matrix $A(U)=\nabla_U F(U)$. Let $\bar{U}\in \mathbb{R}^{n}$ be a constant vector. We rewrite the perturbation  $U=\bar{U}+V$
so that
\begin{equation}
\p_t V+A(\bar{U}+V)\p_xV=0,\quad V|_{t=0}=V_{0}. \label{system}%
\end{equation}
We assume the structural assumption that there exist linearly independent
eigenvectors $r_{i}(\bar{U}+V)$ $(1\leq i\leq n)$ with $\lambda_{1}(\bar
{U}+V)\leq\lambda_{2}(\bar{U}+V)\leq\cdots\leq\lambda_{n}(\bar{U}+V)$ such that
for $1\leq i\leq n,$
\begin{equation}
Ar_{i}=\lambda_{i}r_{i},\text{ }\,\text{ }r_{i}\text{ and }\lambda_{i}\in
C^{\infty}\text{ near }V=0. \label{eigen}%
\end{equation}
We then define $T=(r_{i})_{i=1}^{n}$ so that
\begin{equation}
A=T\text{diag}(\lambda_{1},\lambda_{2},\dots,\lambda_{n})T^{-1}%
\equiv:T\Lambda T^{-1}. \label{inverse}%
\end{equation}
Such a smoothness assumption \eqref{eigen} is valid if $A$ is strictly hyperbolic with
distinct eigenvalues $\lambda_{1}<\lambda_{2}<\cdots<\lambda_{n}$, while it is
also valid for the ideal compressible MHD system which may not be strictly hyperbolic.

We define \textit{characteristic unknowns} $f=(f_{i})_{i=1}^n$:
\begin{equation}
f=T^{-1}V\text{ or }V=Tf. \label{fi}%
\end{equation}
We rewrite $A\p_xV=T\Lambda T^{-1}\p_xV$ and diagonalize (\ref{system}) via
multiplying with $T^{-1}$:
\begin{equation}\label{diag}
\lbrack\partial_{t}+\Lambda\partial_{x}]f=[\partial_{t}+\Lambda\partial
_{x}]T^{-1}V, %
\end{equation}
which has one loss of derivative count$.$

We take $\partial_{t}$ of (\ref{system}) to obtain
\begin{align}\label{vt}
[\p_{t}+A\p_{x}]\p_t V  &  =-\p_{t}A\p_{x}V.
\end{align}
We also denote the diagonalization for $\p_{t}V$ (a good unknown for $\p_{t}f$)
as
\begin{equation}
g:\equiv T^{-1}\p_{t}V  =\p_{t} f +T^{-1}\p_{t} T f \label{gi}%
\end{equation}
so that $g$ solves (with no loss of derivative count)
\begin{equation}
\lbrack\partial_{t}+\Lambda\partial_{x}]g=[\partial_{t}+\Lambda\partial
_{x}]T^{-1}\p_{t} V -T^{-1}\p_{t} A \p_{x} V . \label{gequation}%
\end{equation}

$\bullet$ \textit{Inflow Boundary Condition. }We denote the boundary
$
\gamma=\{-1,1\}
$ and the outward normal $\nu_{x}=x$ at $x\in \gamma$. In case of
non-characteristic BC with $\bar{\lambda}_{i}=\lambda_{i}(\bar{U})\neq0,$ we define
the outgoing set $\gamma_{+}^{i}$ and incoming set $\gamma_{-}^{i}$ associated with the linearized $i$-th family
transport operator $\partial_{t}+\bar{\lambda}_{i}\partial_{x}$ for the
characteristic unknown $f_{i}$ as
\beq
\gamma_{+}^{i}    =\{x\in\gamma: \bar{\lambda
}_{i}\nu_{x}>0\},\quad
\gamma_{-}^{i}    =\{x\in\gamma: \bar{\lambda
}_{i}\nu_{x}<0\}. \label{gamma}%
\eeq
We impose the basic inflow BC for each $f_{i}$ as%
\begin{equation}
f_{i}|_{\gamma_{-}^{i}}=b_{i},\ 1\le i\le n. \label{inflow}%
\end{equation}

Our first main result is

\begin{theorem}\label{nbyn}
Assume (\ref{eigen}), the non-characteristic condition:
\beq\label{lami}
\min_{1\leq i\leq n}|\lambda_{i}(\bar{U})|\neq0
\eeq
and the compatibility condition:
\begin{equation}
f_{0,i}|_{\gamma_{-}^{i}}=b_{i}(0),\ 1\le i\le n. \label{comp}%
\end{equation}
There exists a constant  $\eps_0>0$ such that if
\begin{equation}
\norm{V_{0}}_{W^{1,\infty}(I)}+ \norm{b}_{W^{1,\infty}(\mathbb{R}_+)}\le \eps_0, \label{small}%
\end{equation}
then there is a unique solution to \eqref{system} with inflow BC \eqref{inflow}  on $[0,\infty)$ satisfying
\begin{align}
\sup_{0\le t<\infty}\sum_{j=0}^1 \norm{\p_t^j V(t)}_{W^{1-j,\infty}(I)}+ \norm{ V|_\gamma}_{W^{1,\infty}(\mathbb{R}_+)}
\lesssim \norm{V_{0}}_{W^{1,\infty}(I)}+\norm{b}_{W^{1,\infty}(\mathbb{R}_+)} . \label{vdecay}%
\end{align}

\end{theorem}
We use $A\ls B$ to denote $A\le CB$ for generic constants $C>0$. In \eqref{vdecay}, the trace
at the boundary $\gamma$ of $\p_tV\backsim g$ is well-defined in the transport theory in
view of (\ref{gequation}). The inflow data $b_i(t)$ in (\ref{inflow}) can
be any bounded smooth function with a small amplitude for all $t\geq0,$ which
plays a crucial role in practice for carrying out numerical simulations in
finite domains. Our result requires only the structural condition
(\ref{eigen}), which covers important examples such as elastodynamics,
compressible Euler system and ideal compressible MHD system.

Even though higher regularity (such as $C^{1}$) is expected with more
compatibility conditions at incoming set $\gamma_{-}^{i}$, this is beyond the
focus of our current study with a minimum assumption. In fact, our
$W^{1,\infty}$ stability rules out any possible shock, rarefaction and
contact discontinuity formation for initial small $W^{1,\infty}$ perturbation
of $\bar{U}.$ This is already in stark contrast to the Cauchy problem for
\eqref{conser}, for which it is well-known that shock will develop for
generic small smooth perturbation of $\bar{U}$ \cite{J}. Since there is an explosive literature on the BV theory of Cauchy problem for
(\ref{conser}) in 1D, it is impossible to select a fair representative reference
list. We cite only related work for general systems \cite{A,BB,BCP,BLY,D,DM,G,Goodman}, and refer the
comprehensive reference book \cite{Dbook} for more references.


\subsection{3D Incompressible Euler System}

We study the 3D incompressible Euler system in a finite pipe $\Omega=(-1,1)\times D \subseteq \mathbb{R}^3$ with unit square $D=(-1,1)^2$ or unit disk $D=B(0,1)$:
\begin{equation}\label{euler}
\begin{cases}
\partial_t  u + u\cdot\nabla  u+  \nabla  p = 0 &\text{in }\Omega,\\
{\rm div}  u = 0&\text{in }\Omega,
\end{cases}
\end{equation}
where $ u$ is the velocity field and $ p$ is the scalar pressure. 

We consider the 3D stability of shear flows:
\beq
u_s=(U,0,0)\text{ with }U=U(x_2,x_3)>0,
\eeq
and denote its vorticity
\beq\label{omegas}
\omega_s:={\rm curl}u_s=(0,\p_{3}U,-\p_{2}U).
 \eeq
 We rewrite the perturbation $ u=u_s+v$ (and $\curl u=\omega_s+\omega$), and consider the following inflow problem:
\begin{equation}\label{eulerv}
\begin{cases}
\partial_t v +(u_s+v)\cdot\nabla v + v \cdot\nabla  u_s +  \nabla  p  = 0 \quad\text{in }\Omega,\\
{\rm div} v = 0\qquad\qquad\qquad\qquad\qquad\qquad\qquad\,\text{in }\Omega,
\\ v|_{t=0}=v_0
\end{cases}
\end{equation}
with the boundary conditions
\beq\label{eulervBC}
v\cdot \nu=0 \text{ on }\gamma_0,\  v_1=v_{b,1}^\pm \text{ on }\gamma_\pm,\   \omega =\omega_b^- \text{ on }\gamma_-,
\eeq
where we have denoted the outflow and inflow boundaries and the lateral one by
\beq
\gamma_\pm:=\{\pm1\}\times D ,\quad \gamma_0:=[-1,1]\times \partial D,
\eeq
respectively, with $\nu$ the unit outward normal to $\gamma_0$. 
We note that the inflow BC for the vorticity perturbation $\omega$ in \eqref{eulerv} is naturally imposed in terms of the vorticity (transport) equations:
\begin{align}\label{omegaeq}
\begin{cases}
\p_t\omega+(u_s+v)\cdot\nabla  \omega=- v \cdot\nabla   \omega_s+  \omega_s \cdot \nabla  v+  \omega \cdot \nabla  (u_s +   v)\quad\text{in }\Omega,
\\ \omega|_{t=0}=\omega_0:=\curl v_0,\quad \omega|_{\gamma_-} =\omega_b^-.
\end{cases}
\end{align}

$\bullet$ \textit{Boundary Data. }We assume that the boundary data $v_{b,1}^\pm$ and $\omega_b^-$ satisfy
\beq\label{bbcondition}
 \int_{\gamma_-}v_b^-= \int_{\gamma_+}v_b^+
 \eeq
 and
\begin{equation}\label{bbcondition1}
\omega_{b,1}^-=0,\quad
\p_{2}(U\omega_{b,2}^-+v_b^-(\p_3U+\omega_{b,2}^-))+\p_{3}(U\omega_{b,3}^-+v_b^-(-\p_2U+\omega_{b,3}^-))=0\text{ on }\gamma_-.
\end{equation}
It is remarked that \eqref{bbcondition} is required for the divergence free condition of $v$, and \eqref{bbcondition1} guarantees that $\dive \omega=0$ on $\gamma_-$ (and thus in $\Omega$) in constructing solutions by utilizing the vorticity formulation \eqref{omegaeq}, see \cite{GKM}, where $v$ is uniquely determined by $\omega$ via the $div$--$curl$ system:
\begin{equation}\label{omegaeqv}
\begin{cases}
\curl v=\omega ,\quad \dive v =0\quad\!\text{ in }\Omega,
\\\dis v\cdot \nu=0 \text{ on }\gamma_0,\  v_1=v_{b,1}^\pm \text{ on }\gamma_\pm.
\end{cases}
\end{equation}

Our second main result is

\begin{theorem}\label{th2}
Let $3<p<\infty$. Assume \eqref{bbcondition}, \eqref{bbcondition1} and the compatibility conditions:
\begin{equation}\label{thcom}
\dive v_0=0\text{ in }\Omega,\ v_0\cdot \nu=0\text{ on }\gamma_0,\  v_{0,1}=v_{b,1}^\pm(0)\text{ on }\gamma_\pm,\
\omega_0=\omega_b^-(0)\text{ on }\gamma_-.
\end{equation}
(i) for $D=(-1,1)^2$, assume further
\begin{equation}\label{oscondition31}
\p_{i}U=0\text{ on }\{x_i=\pm1\}\cap \gamma_0,\quad \p_i v_{b,1}^\pm=0\text{ on }\gamma_-\cap\{x_i=\pm1\},\ i=2,3
\end{equation}
and
\begin{equation}\label{oscondition310}
\omega_0\times \nu=0\text{ on }\gamma_0,\quad  \omega_b^- \times \nu =0\text{ on } \gamma_-\cap\gamma_0  ;
\end{equation}
(ii) for $D=B(0,1)$, assume further
\begin{equation}\label{oscondition32}
\nabla U=0\text{ on }\gamma_0,\quad \nabla v_{b,1}^\pm=0\text{ on }\gamma_-\cap\gamma_0 
\end{equation}
and
\begin{equation}\label{oscondition320}
\omega_0 =0\text{ on }\gamma_0,\quad  \omega_b^-  =0\text{ on }   \gamma_-\cap\gamma_0 .
\end{equation}
There exists a constant  $\delta>0$ such that if
\begin{equation}
\norm{\nabla U}_{W^{2,p}(\Omega)}< \delta \frac{\min_{\bar\Omega} U^2}{1+ \min_{\bar\Omega} U},
\end{equation}
then there exists a constant  $\eps_0>0$, depending on $U$, so that if
\begin{align}
\norm{v_0}_{W^{2,p}(\Omega)}+\sup_{0\le t<\infty} \sum_{j=0}^1 \norm{\partial_t^j v_{b,1}^\pm(t)}_{W^{2-j-1/p,p}(\gamma_\pm)}+\sum_{j=0}^1\norm{\partial_t^j\omega_b^-}_{L^\infty(0,\infty;W^{1-j,p}(\gamma_-))} \le \varepsilon_0, \label{small2}%
\end{align}
then there is a unique solution to \eqref{eulerv} and \eqref{eulervBC} on $[0,\infty)$ satisfying
\begin{align}\label{goodes}
&\sup_{0\le t<\infty} \sum_{j=0}^{1}\norm{\p_t^j v(t)}_{W^{2-j,p}(\Omega)}
\nonumber\\&\quad \lesssim \norm{v_0}_{W^{2,p}(\Omega)}+\sup_{0\le t<\infty} \sum_{j=0}^1 \norm{\partial_t^j v_{b,1}^\pm(t)}_{W^{2-j-1/p,p}(\gamma_\pm)}+\sum_{j=0}^1\norm{\partial_t^j\omega_b^-}_{L^\infty(0,\infty;W^{1-j,p}(\gamma_-))} .
\end{align}
\end{theorem}

We note that our theorem is valid for any constant velocity profile, $i.e.,$  $U\equiv const>0$. With higher initial compatibility conditions and higher vanishing lateral conditions, higher regularity of the solutions is expected. However, this is beyond the focus of this paper.

Our new notion and mechanism of stability leads to 3D stability of a large class of shear flows beyond the Arnold's celebrated stability in 2D \cite{arnold}, leading to a rare construction of a global unique regular solution for the 3D incompressible Euler system \cite{C}. This is in stark contrast to recent breakthroughs both in the study of formation of singularities \cite{CH0,CH1,CH,E,EGM,WLGB} (see the survey \cite{DE} for more references), and in the study of (non-unique) weak solutions \cite{BDLIS15,DS,DS13,GKN,GR,I18,NV,Sc,Sh} (see the book \cite{BMNV} for more references). We also recall recent construction of global smooth incompressible Euler flows in the presence of rotation \cite{GPW,RT}.

Moreover, our new stable shear flows can have rather arbitrary velocity profiles (up to a constant) except with a small gradient and vanishing conditions \eqref{oscondition31} or \eqref{oscondition32} for 3D square or circular pipes, respectively, which could serve as an (local) attractor, exhibiting a rich and complex dynamics in the inviscid limit of the Navier--Stokes flows in the presence of inflow type of BC.
This is in sharp contrast to classical instability study for shear flows initiated by Lord Rayleigh \cite{R} under the $x_1$-periodic BC, see \cite{Lin1,Lin2}. Particularly, the classical unstable profile of the shear flow  $(\cos m x_2+c,0)$ with inflection points (see \cite{FH}) becomes {\it stable} in a 3D square pipe in the presence of inflow BC for $c$ sufficiently large, with $\omega_s=(0,0,m\sin m x_2)$ vanishes at the lateral boundary $\{x_2=\pm \pi\}$. For a 2D channel no vanishing condition for the scalar $\omega_s$ is needed, see the recent work \cite{GY}, and the same remark goes also for a 3D circular pipe of axisymmetric flow with no swirl. We also recall the breakthrough \cite{BM} in the study of nonlinear inviscid damping of monotone shear flows for the 2D incompressible Euler system under the $x_1$-periodic BC in Gevrey spaces, and more recent advances \cite{IJ,MZ} among others.

The local well-posedness of inflow problem for the vorticity formulation for 3D incompressible Euler system in a regular bounded domain has been recently established in H\"older spaces in Theorem 1.4 of \cite{GKM}. In order to obtain global stability, we construct the solutions in $W^{1,p}$ spaces instead for $p>3$. The subtle technical issue for nonlinear closure is to establish the $W^{2,p}$ regularity estimate in Theorem \ref{vomth} for the {\it div$-$curl} system \eqref{omegaeqv} for solving the velocity field $v$ from the vorticity $\omega$, which is invalid in general in either square or circular pipes for $p>3$ due to their non-smoothness nature at the corner $\gamma_\pm\cap\gamma_0$. Even though it is natural to impose further compatibility at  $\gamma_\pm\cap\gamma_0$ to ensure $W^{2,p}$ estimate, it is impossible for inflow BC as the outflow boundary $\gamma_+$ can not be prescribed. Instead, we seek regular extension over entire lateral surface $\gamma_0$ to avoid the corner at $\gamma_\pm$. To this end, we imposed the crucial vanishing conditions for $U$ as in \eqref{oscondition31} and \eqref{oscondition32}. Note that the normals for square pipe are given by $\nu=\pm e_i$ on $\{x_i=\pm 1\}\cap\gamma_0$, $i=2,3$, from \eqref{omegas} we deduce that \eqref{oscondition31} and \eqref{oscondition32} are equivalent to
\beq\omega_s\times \nu=0\text{ and }\omega_s=0\text{ on }\gamma_0
 \eeq
 for square and circular pipes, respectively. Our key observation is that the vorticity equations are invariant on $\gamma_0$ thanks to the non-penetration BC (Lemma \ref{inlem}), which allows us to enforce important vanishing properties for the vorticity $\omega$:
 \beq
\omega\times \nu=0\text{ and }\omega=0\text{ on }\gamma_0
\eeq
for square and circular pipes, respectively, see part (3) of Theorem \ref{vorth}. Combining with the exact geometric structure of square and circular pipes, we can use precise parity extension across the lateral surface leading to higher Sobolev regularity for $p>3$, see part (2) of Theorem \ref{vomth}.

\section{1D Hyperbolic Conservation Laws}

\subsection{Dissipation for Linear Transport Equation}

We begin with the transport inflow problem with $|\lambda|\neq0$:%
\begin{align}
&\partial_tf+\lambda \partial_x f=h,\label{transport}\\
&f|_{t=0}=f_{0},\quad f|_{\gamma_{-}}=b.\label{initial}
\end{align}

Fix $t\geq0$ and $ x\in \bar I,$ we define the characteristic curve $X(\tau;t,x)$ passing through $(t,x)$:
\begin{equation}
\frac{d}{d\tau}X(\tau;t,x)=\lambda(\tau,X(\tau;t,x)),\quad X(t;t,x)=x. \label{char}%
\end{equation}
We define the backward exit-time for $(t,x)$ as
\[
t_{b}^{-}(t,x)=\inf_{s\leq t}\{   X(s;t,x)\in \bar I\}.
\]
We note that $0\leq t_{b}^{-}(t,x)\leq t$, and $t_{b}^{-}(t,x)=t$ if and only if $(t,x)\in\{t\ge0,x\in\gamma_{-}\}\cup\{t=0,x\in \bar I\}$. We also define the
backward exit point $x_{b}^{-}(t,x)=X(t_{b}^{-}(t,x);t,x)$ such that $(t_{b}^{-},x_{b}^{-})\in\{t\ge0,x\in\gamma_{-}\}\cup\{t=0,x\in\bar I\}.$ See \cite{GKTT} for general discussion for kinetic transport equations.

\begin{lemma}\label{ode}\bigskip(Backward exit time $t_{b}^{-}$)
Assume the Lipschitz bound $\norm{\lambda}_{W^{1,\infty}(\mathbb{R}_+\times I)}<+\infty,$ and the non-characteristic condition:
\begin{equation}
\inf_{0\leq t<\infty,  x\in\bar I}|\lambda(t,x)|>0. \label{non}%
\end{equation}
Then $t_{b}^{-}(t,x)$ is unique. If $\lambda>0,$ the characteristic curve $\Gamma\equiv\{(\tau,X(\tau;0,-1)),\tau\geq0\}$, emanating from the inflow corner $(0,-1)$, splits $\{t\geq0,\ x\in\bar I\}$ into disjoint sets:
\begin{align}
\{(t,x)  &  :t_{b}^{-}(t,x)=0,\  x_{b}^{-}(t,x)>-1\}=\{(t,x):x>X(t;0,-1)\},\nonumber\\
\{(t,x)  &  :t_{b}^{-}(t,x)>0,\  x_{b}^{-}%
(t,x)=-1\}=\{(t,x):x<X(t;0,-1)\},\label{split}\\
\{(t,x)  &  :t_{b}^{-}(t,x)=0,\  x_{b}^{-}%
(t,x)=-1\}=\{(t,x):x=X(t;0,-1)\}\equiv \Gamma.\nonumber
\end{align}
If $\lambda<0,$ the characteristic curve $\Gamma\equiv\{(\tau,X(\tau;0,1)),\tau\geq0\}$, emanating from the inflow corner $(0,1)$, splits $\{t\geq0,  x\in\bar I\}$ into three disjoint sets:
\begin{align}
\{(t,x)  &  :t_{b}^{-}(t,x)=0,\  x_{b}^{-}%
(t,x)<1\}=\{(t,x):x<X(t;0,1)\},\nonumber\\
\{(t,x)  &  :t_{b}^{-}(t,x)>0,\  x_{b}^{-}%
(t,x)=1\}=\{(t,x):x>X(t;0,1)\},\label{split2}\\
\{(t,x)  &  :t_{b}^{-}(t,x)=0,\  x_{b}^{-}%
(t,x)=1\}=\{(t,x):x=X(t;0,1)\}\equiv \Gamma.\nonumber
\end{align}
Moreover, $t_{b}^{-}(t,x)>0\in
W^{1,\infty}$ up to $\Gamma,$ and $-1<x_{b}^{-}(t,x)<1\in W^{1,\infty}$ up to
$\Gamma.$
\end{lemma}

\begin{proof}
Without loss of generality, we may assume $\lambda>0$ so that $\gamma
_{-}=\{-1\}.$ It follows from (\ref{char}) and $\lambda>0$ that $t_{b}^{-}(t,x)$ is unique. By
the uniqueness of ODE (\ref{char}), it then follows that $X(\tau
;t,x)-X(\tau;0,-1)>0$ and $X(\tau;t,x)-X(\tau;0,1)<0$ are invariant sets. By
inspection of cases of backward exits points $(t_{b}^{-},x_{b}^{-})$, we
conclude that the splitting (\ref{split}) is valid.

We next show $W^{1,\infty}$ regularity for back-time exit time $t_{b}^{-}$
and point $x_{b}^{-}.$

If $t_{b}^{-}(t,x)=0$ and $x_{b}^{-}(t,x)=X(0;t,x)>-1,$ it
follows from the continuity of ODE (\ref{char}) with respect to $(t,x)$ that $t_{b}^{-}(t^{\prime},x^{\prime})=0$ for $(t^{\prime},x^{\prime})$ near
$(t,x)$ so that $x_{b}^{-}(t^{\prime},x^{\prime})=X(0;t^{\prime},x^{\prime
}).$ Clearly, $x_{b}^{-}(t,x)\in W^{1,\infty}$ from the standard ODE theory,
uniformly up to $\Gamma$.

If $t_{b}^{-}(t,x)>0 $ and $x_{b}^{-}(t,x)=-1,$ then we have
\[
x+\int_{t}^{t_{b}^{-}}\lambda(\tau,X(\tau;t,x))d\tau=-1.
\]
We view $t_{b}^{-}$ as the zero point of the $C^{1}$ function $F_{-}(\theta)$:
\[
F_{-}(\theta)\equiv x+\int_{t}^{\theta}\lambda(\tau,X(\tau;t,x))d\tau+1,
\]
with $\frac{\partial F_{-}(\theta)}{\partial\theta}|_{\theta=t_{b}^{-}}%
\equiv\lambda(\theta,X(\theta;t,x))|_{\theta=t_{b}^{-}}>0.$ We then deduce
from the Implicit Function Theorem that $t_{b}(t,x)\in W^{1,\infty},$
uniformly up to $\Gamma$ thanks to (\ref{non}).
\end{proof}

\begin{theorem}
\label{scalar}\bigskip Assume (\ref{non}) and define the spatial weight
function:
\begin{equation}
w(x):=e^{-\alpha\frac{\lambda}{|\lambda|}x}=e^{-\alpha\text{sgn}%
 \lambda x}. \label{w}%
\end{equation}
(1) Assume
$\norm{\lambda}_{W^{1,\infty}(\mathbb{R}_+\times I)}+\norm{f_0}_{L^\infty(I)}+\norm{b}_{_{L^\infty(\mathbb{R}_+)}}+ \norm{h}_{L^\infty(\mathbb{R}_+\times I)}<+\infty.
$
Then there exists a unique solution $f$ with $f,f|_{\gamma_+}\in L^\infty$ to \eqref{transport} and
\eqref{initial} such that
\begin{equation}
\sup_{0\le t<\infty} \norm{wf(t)}_{L^\infty(I)}+ \norm{w f|_{\gamma}}_{L^\infty(\mathbb{R}_+)}     \ls  \norm{wf_0}_{L^\infty(I)}+\norm{wb}_{_{L^\infty(\mathbb{R}_+)}}+\norm{wh}_{L^\infty(\mathbb{R}_+\times I)}. \label{infinity}%
\end{equation}
(2) Assume
$\norm{\lambda}_{W^{1,\infty}(\mathbb{R}_+\times I)}+\norm{f_0}_{W^{1,\infty}(I)}+\norm{b}_{W^{1,\infty}(\mathbb{R}_+)}+ \sum_{j=0}^1\norm{\p_t^jh}_{L^\infty(\mathbb{R}_+\times I)} <+\infty
$
and the compatibility condition $
f_{0}|_{\gamma_{-}}=b(0).$ Then  $f,f|_{\gamma_+}\in W^{1,\infty}$, and for some
$\alpha>0,$
\begin{align}
&\sup_{0\le t<\infty} \norm{w\p_t f(t)}_{L^\infty(I)}+ \norm{w\p_t f|_{\gamma}}_{L^\infty(\mathbb{R}_+)}
\nonumber\\&\quad\ls  \norm{w\p_tf_0}_{L^\infty(I)}+\norm{w\p_tb}_{L^{\infty}(\mathbb{R}_+)}+\sum_{j=0}^1\norm{w\p_t^jh}_{L^\infty(\mathbb{R}_+\times I)},\label{lineardecay}
\end{align}
where $\partial_{t}f_{0}=-\lambda|_{t=0} \p_xf_{0}+h|_{t=0}.$
\end{theorem}

The spatial weight $w$ in (\ref{w}) induces the key exponential decay to
control the time integration for the forcing $h,$ so that a nonlinear closure
is possible later. We remark that $\p_tf$ is the key quantity to control, as
$\p_xf$ can be solved directly from (\ref{transport}) as long as $\lambda
\neq0.$

\begin{proof}
Without loss of generality, we assume $\lambda>0.$ Recall Lemma \ref{ode}. For
 $t\ge 0$ and $x\in  \bar I,$ we define
\begin{align}\label{expression}
\!f(t,x)=
\begin{cases}
\dis f_{0}(X(0;t,x))+\int_{0}^{t}h(\tau;X(\tau;t,x))d\tau,\quad\text{if
}t_{b}^{-}(t,x)=0\text{ and } x_{b}^{-}%
(t,x)>-1,\\
\dis b(t_{b}^{-}(t,x))+\int_{t_{b}^{-}(t,x)}^{t}h(\tau;X(\tau;t,x))d\tau,\quad\!\!\text{if
}t_{b}^{-}(t,x)>0\text{ and } x_{b}^{-}%
(t,x)=-1.
\end{cases}
\end{align}
From \eqref{split}, it follows that $f(t,x)$ is well-defined mild solution to \eqref{transport}
except along the 1D $W^{1,\infty}$ curve $\Gamma=\{(\tau,X(\tau;0,-1)),\tau\geq0\}$, while the
trace $f|_{\gamma_{+}}$ (letting $x\in \gamma_+=\{1\}$ in \eqref{expression}) is well-defined except the intersection point of $\Gamma$ with $[0,\infty)\times \gamma_+$.

$\bullet$ \textit{Proof of Part (1).} Recall the weight function ($\lambda
\neq0$) from (\ref{w}) such that
\[
w'=-\alpha\frac{\lambda}{|\lambda|} w=-\alpha\, \text{sgn}\lambda\, w.
\]
We denote $\lambda_{m}=\inf|\lambda|>0.$ Then we have from (\ref{transport}), away from $\Gamma,$
\begin{align*}
[\p_t +\lambda\p_x]\{e^{\alpha\lambda_{m}t}%
wf\} +(|\lambda|-\lambda_{m})\alpha\{e^{\alpha\lambda_{m}t}wf\}  &
=e^{\alpha\lambda_{m}t}wh.
\end{align*}
We approximate $|f|$ by $\sqrt{f^{2}+\kappa}$ as $\kappa\rightarrow0,$ and
multiply with $\partial_{f}\sqrt{f^{2}+\kappa}=\frac{f}{\sqrt{f^{2}+\kappa}}$
to obtain
\begin{align}\label{fabsolute}
  [\p_t  +\lambda\p_x]\{e^{\alpha
\lambda_{m}t}w\sqrt{f^{2}+\kappa}\} +(|\lambda|-\lambda_{m})\alpha
\{e^{\alpha\lambda_{m}t}w\sqrt{f^{2}+\kappa}\}
  =e^{\alpha\lambda_m t}wh\frac{f}{\sqrt{f^{2}+\kappa}}.
\end{align}
Since $|\lambda|-\lambda_{m}\geq0$ and $\frac{|f|}{\sqrt{f^{2}+\kappa}}\leq1,$
we integrate \eqref{fabsolute} along the backward characteristic curve (\ref{char}) and
(\ref{expression}) with $X(t;t,x)=x$ to obtain, away from $\Gamma,$
\begin{align}
e^{\alpha\lambda_{m}t}w(x)\sqrt{f^{2}(t,x) +\kappa}%
 \leq& \max\left\{w(X(0))\sqrt{f_0^{2}(X(0))+\kappa},\ e^{\alpha\lambda_{m}t_{b}}w(-1)\sqrt{f^{2}(t_{b},-1)+\kappa}%
\right\}\nonumber \\
&  +\int_{0}^{t}e^{\alpha\lambda_m\tau}w( X(\tau))|h(\tau,X(\tau
))|d\tau.\nonumber
\end{align}
Dividing $e^{\alpha\lambda_{m}t}$ and taking $\kappa\rightarrow0$, from
$t_{b}\leq t,$ we obtain
\begin{align}
w(x)|f(t,x)|  &  \leq\max\left\{\norm{wf_{0}}_{L^\infty(I)},\norm{wb}_{L^\infty(\mathbb{R}_+)}\right\}+\int_{0}%
^{t}e^{-\alpha\lambda_{m}(t-\tau)}w( X(\tau))|h(\tau,X(\tau
))|d\tau\nonumber\\
&  \leq \norm{wf_{0}}_{L^\infty(I)}+\norm{wb}_{L^\infty(\mathbb{R}_+)}+\frac{1}{\alpha\lambda_{m}%
}\norm{wh}_{L^\infty(\mathbb{R}_+\times I)},\label{identity2}
\end{align}
where the crucial decay from $e^{\alpha\lambda_{m}t}$ yields
\[
\int_{0}^{t}e^{-\alpha\lambda_{m}(t-\tau)}d\tau=\frac{1}{\alpha\lambda_{m}%
}.
\]
For each $t\ge 0$ taking $L^\infty$ over $x\in I$ of \eqref{identity2} and for $x=1$ taking $L^\infty$ over $t\in\mathbb{R}_+$, we then conclude (\ref{infinity}).

The mild solution $f$ with trace $f_{\gamma_{+}}$ defined in (\ref{expression}) satisfies the weak formulation of
\[
\langle f,\phi\rangle|_{t=T}-\langle f_0,\phi\rangle|_{t=0}-(f,[\partial_{t}+\lambda\partial_{x}%
]\phi)-(f,\partial_{x}\lambda\phi)+\langle\lambda f_{\gamma_{+}},\phi\rangle|_{(0,T)\times\gamma
_{+}}-\langle\lambda b,\phi\rangle|_{(0,T)\times\gamma_{-}}=(f,h)
\]
for any test function such that $\phi$, $[\partial_{t}+\lambda\partial_{x}]\phi,$ $\phi|_{\gamma_{\pm}}$ and $\phi|_{t=0},\phi|_{t=T}\in L^1$. Here $(\cdot
,\cdot)$\ and $\langle\cdot,\cdot\rangle$ are $L^{2}$ inner product in the
bulk, and at boundary $t=T,$ $t=0$, $(0,T)\times\gamma_{+} $ and
$(0,T)\times\gamma_{-} $, respectively. To show uniqueness within the
class of $f,f_{\gamma_{+}}\in L^\infty$, we may
take $b\equiv f_{0}\equiv h\equiv0.$ We solve a dual problem to construct a
test function for $0\leq t\leq T,$
\begin{align*}
&\lbrack\partial_{t}+\lambda\partial_{x}]\phi_{f}+\partial_{x}\lambda\phi_{f}
 =0,\\
&\phi_{f}|_{t=T}   =f(T),\quad \phi_{f}|_{\gamma_{+}}=f_{\gamma_{+}}.
\end{align*}
We then deduce uniqueness by testing against such a $\phi_{f}$: $\|f(T)\|_{L^2(I)}%
^{2}+\int_{0}^{T}\lambda f_{\gamma_{+}}^{2}=0.$

$\bullet$ \textit{Proof of Part (2).} We first establish $f\in W^{1,\infty},$
and then obtain $\partial_{t}f$ estimate. Since $b,h\in W^{1,\infty}$ so that
$f_{0}|_{x=-1}$ and $b|_{t=0}$ are well-defined and the compatibility
condition is now imposed from (\ref{initial}) at $t=0:$
\begin{equation}
f_{0}(-1)=b(0). \label{compatible}%
\end{equation}
Recall (\ref{expression}). Clearly, if $t_{b}^{-}>0$, by Lemma \ref{ode},
$t_{b}^{-}\in W^{1,\infty},$ so that $b(t_{b}(t,x))$ is well-defined near
$(t,x)$ and $b(t_{b})\in W^{1,\infty}$ and so is $f$ near $(t,x);$ if
$t_{b}^{-}=0$ and $x_{b}^{-}>-1,$ then $f_{0}(X(0;t,x))$ is well-defined near
$(t,x)$ so $f\in W^{1,\infty}$ near $(t,x).$ It follows from (\ref{split})
that $f\in W^{1,\infty}$ except $\Gamma$ but up to $\Gamma$ from both sides.
On the other hand, if $t_{b}^-=0$ and $ x_b^-=X(0;t,x)=-1,$ then we define $f$ from
(\ref{compatible}) via either definition for $t_{b}^->0$ or $t_{b}^-=0$ in
(\ref{expression}):
\begin{align*}
f(t,x)  &  =f_{0}(X(0;t,x))+\int_{0}^{t}h(\tau;X(\tau;t,x))d\tau\\
&  =f_{0}(-1)+\int_{0}^{t}h(\tau;X(\tau;t,x))d\tau=b(0)+\int_{0}^{t}%
h(\tau;X(\tau;t,x))d\tau.
\end{align*}
We now vary $(t,x)$
for $t\geq0\,\ $and $x\in\bar I $ to see that $f\in W^{1,\infty}$ except along the 1D
$W^{1,\infty}$ curve $\Gamma$ while continuous
along this curve. We then conclude that $f\in W^{1,\infty}$. See \cite{GKTT} for
more detailed discussion for kinetic equations.

We next take $\p_t$ of \eqref{transport} to obtain%
\begin{align*}
&\lbrack\partial_{t}+\lambda\partial_{x}]\partial_{t}f+\p_t\lambda\partial
_{x}f=\p_t h,
\\&
\partial_{t}f|_{t=0}=-\lambda|_{t=0}\partial_{x}f_{0}+h|_{t=0},\quad\partial_{t}f|_{\gamma_{-}}=\p_tb.
\end{align*}
We also solve from \eqref{transport} for $\lambda\neq0$:
\begin{equation}
\partial_{x}f=\frac{1}{\lambda}\{h-\partial_{t}f\}\nonumber
\end{equation}
so that $\partial_{t}f$ satisfies
\beq
\lbrack\partial_{t}+\lambda\partial_{x}]\partial_{t}f+\frac{\p_t\lambda
}{\lambda}\{h-\partial_{t}f\}=\p_t h. \label{fx}
\eeq
We then apply the estimate (\ref{identity2}) to \eqref{fx} with a new forcing $\p_th-\frac
{\p_t\lambda}{\lambda}\{h-\partial_{t}f \}$:
\begin{align*}
& \sup_{0\le t<\infty} \norm{w\p_tf(t)}_{L^\infty(I)}+ \norm{w\p_tf|_\gamma}_{L^\infty(\mathbb{R}_+)}  \\
& \quad \leq \norm{w\p_tf_{0}}_{L^\infty(I)}+\norm{w\p_tb}_{L^{\infty}(\mathbb{R}_+)}+\frac{1}{\alpha\lambda_{m}}\norm{w\p_th}_{L^\infty(\mathbb{R}_+\times I)}
\\
& \qquad +\frac{1}{\alpha\lambda_{m}}\norm{\frac
{\p_t\lambda}{\lambda}}_{L^\infty(\mathbb{R}_+\times I)}\left\{ \norm{w h}_{L^\infty(\mathbb{R}_+\times I)}+\norm{w\partial
_{t}f}_{L^\infty(\mathbb{R}_+\times I)}\right\}.
\end{align*}
Choosing\ $\alpha\lambda_{m}=2\norm{\frac
{\p_t\lambda}{\lambda}}_{L^\infty(\mathbb{R}_+\times I)}+2,$ we
conclude (\ref{lineardecay}).
\end{proof}

\subsection{Hyperbolic Systems}

To avoid the loss of derivative count (cf. \eqref{diag}) in constructing solutions to the nonlinear system \eqref{system},
we define the iteration by solving $V^{l+1}$ for the following smooth approximation:
\begin{equation}\label{nsystem}
\p_tV^{l+1}+A_{l}^{\ast}\p_xV^{l+1}=0,\quad V^{l+1}|_{t=0}={T_{l}^\ast}|_{t=0}T^{-1}_0V_{0},
\end{equation}
where
\begin{equation}\label{star}
V^l_\ast=\eta_{\frac{1}{l}}\ast\widetilde{V^l},\  A_{l}^{\ast}=A(\bar{U}+V_{\ast}^{l}),\ T_{l}^\ast=T(\bar
{U}+V_{\ast}^{l}),\ \Lambda_{l}^\ast=\Lambda(\bar{U}+V_{\ast}^{l}).
\end{equation}
Here  $\eta$ is the standard mollifier, and the extension $\widetilde{V^l}$ is defined to be trivial $V^l(0,x)$ for $t\leq0$, $V^l(t,-1)$ for $x\leq-1$ and $V^l(t,1)$ for $x\geq1$.

We start with $f^1=0$, and then solve iteratively $f^{l+1}:={T_{l}^\ast}^{-1}V^{l+1}$ from \eqref{nsystem}: for $1\le i\le n$,
\begin{equation}\label{fn}
\begin{cases}
\lbrack\partial_{t}+\lambda_{li}^{\ast}\partial_{x}\rbrack f_{i}^{l+1}=\{[\partial
_{t}+\lambda_{li}^{\ast}\partial_{x}]{T_{l}^\ast}^{-1}T_{l}^\ast f^{l+1}\}_{i},
\\
f_{i}^{l+1}|_{t=0}=f_{0,i},\quad f_{i}^{l+1}|_{\gamma_{-}^{i}}=b_{i} %
\end{cases}
\end{equation}
with the compatibility condition for continuity:
\beq\label{jcom}
f_{0,i}|_{\gamma_{-}^{i}%
}=b^{i}(0).
\eeq

$\bullet$ \textit{Induction Hypothesis:} We assume that
\begin{equation}
\sup_{0\le t<\infty}\sum_{j=0}^1 \norm{\p_t^j V^l(t)}_{W^{1-j,\infty}(I)}+\norm{ V^l|_{\gamma}}_{W^{1,\infty}(\mathbb{R}_+)} \leq \delta\ll1. \label{inductionl}%
\end{equation}
By \eqref{inductionl}, the approximation $V_{\ast}^{l}$ from \eqref{star} satisfies
\begin{align}
\|V_{\ast}^{l}\|_{C^1}\lesssim\delta,
\ \|V_{\ast}^{l}-V^{l}\|_{C^0}\lesssim\frac{1}{l},
\text{ but }\|V_{\ast}^{l}\|_{C^{2}}\lesssim l.  \label{v*}
\end{align}
Then by \eqref{eigen} and \eqref{lami},
\beq\label{iposi}
\min_{1\leq i\leq n}\inf_{0\leq t<\infty,  x\in\bar I}|\lambda_{li}^\ast  |    \geq\frac{1}%
{2}\min_{1\leq i\leq n}|  \lambda_{i}(\bar U)|>0
\eeq
and
\begin{align}\label{lambdal}
&\|A_{l}^{\ast}\|_{C^0}, \|A_{l}^{\ast-1}\|_{C^0},\|T_{l}^{\ast}\|_{C^0}, \|T_{l}^{\ast-1}\|_{C^0} , \|\Lambda^\ast \|_{C^0} \lesssim1,
\\
&\|\p_t A_{l}^{\ast}\|_{C^0},\|\p_x A_{l}^{\ast}\|_{C^0}, \|\p_t T_{l}^{\ast}\|_{C^0}, \|\p_x T_{l}^{\ast}\|_{C^0},\|\p_t\Lambda^\ast \|_{C^0} \lesssim \delta
.\label{lambdal2}
\end{align}

We first show qualitatively $f^{l+1}\in W^{1,\infty}$ for fixed $l$.
\begin{lemma}
Assume \eqref{inductionl} and \eqref{jcom}. There is a unique solution $f^{l+1}$ to \eqref{fn} such that
\beq
\sup_{0\le t<\infty}\norm{f^{l+1}(t)}_{L^\infty(I)}+ \norm{f^{l+1}|_{\gamma}}_{L^\infty(\mathbb{R}_+)}\lesssim \norm{V_0}_{L^\infty(I)}
+\norm{b}_{L^{\infty}(\mathbb{R}_+)}\label{claim}
\eeq
and
\beq
\sup_{0\le t<\infty}\norm{\partial_{t}f^{l+1}(t)}_{L^\infty(I)}+ \norm{\partial_{t}%
f^{l+1}|_{\gamma}}_{L^{\infty}(\mathbb{R}_+)}
\lesssim l\left\{\norm{V_{0}}_{W^{1,\infty}(I)}+\norm{b}_{W^{1,\infty}(\mathbb{R}_+)} \right\}.
\eeq

\end{lemma}

\begin{proof}
For fixed $l,$ we use a further iteration $f_{i}^{l+1,j}$ with $f_{i}^{l+1,1}=0$ to solve $f_{i}^{l+1}$ in \eqref{fn}:
\beq\label{fnj}
\begin{cases}
\lbrack\partial_{t}+\lambda_{li}^{\ast}\partial_{x}\rbrack f_{i}^{l+1,j+1}=\{[\partial
_{t}+\lambda_{li}^{\ast}\partial_{x}]{T_{l}^\ast}^{-1}T_{l}^\ast f^{l+1,j}\}_{i},
\\
f_{i}^{l+1,j+1}|_{t=0}=f_{0,i},\quad f_{i}^{l+1,j+1}|_{\gamma_{-}^{i}}=b_{i}.
\end{cases}
\eeq
Recall \eqref{iposi}. We therefore apply \eqref{infinity} in Theorem \ref{scalar} to \eqref{fnj}: defining $w=w_{i}$ as in \eqref{w} for each $\lambda=\lambda_{li}^{\ast}$ and noting $w_i\sim 1$, by \eqref{lambdal} and \eqref{lambdal2},
\begin{align}\label{j}
&\sup_{0\le t<\infty} \norm{f_{i}^{l+1,j+1}(t)}_{L^\infty(I)}+ \norm{f_{i}^{l+1,j+1}|_{\gamma}}_{L^{\infty}(\mathbb{R}_+)}
  \nonumber\\& \quad  \lesssim \norm{f_{0,i}}_{L^\infty(I)}
+\norm{b_i}_{L^{\infty}(\mathbb{R}_+)} + \sup_{0\le t<\infty}\norm{\{[\partial
_{t}+\lambda_{li}^{\ast}\partial_{x}]{T_{l}^\ast}^{-1}T_{l}^\ast f^{l+1,j}\}_i(t)}_{L^\infty(I)}
 \nonumber\\&  \quad \lesssim \norm{V_0}_{L^\infty(I)}
+\norm{b}_{L^{\infty}(\mathbb{R}_+)} + \delta\sup_{0\le t<\infty}\norm{ f^{l+1,j}(t)}_{L^\infty(I)}.
\end{align}
For $\delta\ll1,$ it follows that from summation over $1\leq i\leq n$,%
\beq\label{fjes}
\sup_{0\le t<\infty} \norm{f^{l+1,j+1}(t)}_{L^\infty(I)}+ \norm{f^{l+1,j+1}|_{\gamma}}_{L^{\infty}(\mathbb{R}_+)} \lesssim \norm{V_0}_{L^\infty(I)}
+\norm{b}_{L^{\infty}(\mathbb{R}_+)}
\eeq
by induction on $j$.
Moreover, $\{f^{l+1,j+1}-f^{l+1,j}\}$ satisfies
\beq\label{jcon}
\begin{cases}
\lbrack\partial_{t}+\lambda_{li}^{\ast}\partial_{x}\rbrack\{f_{i}^{l+1,j+1}-f_{i}^{l+1,j}\}
=\{\lbrack\partial_{t}+\lambda_{li}^{\ast}\partial_{x}\rbrack {T_{l}^\ast}^{-1}T_{l}^\ast\{f^{l+1,j}%
-f^{l+1,j-1}\}\}_i,\\
\{f_{i}^{l+1,j+1}-f_{i}^{l+1,j}\}|_{t=0}=0,\quad   \{f_{i}^{l+1,j+1}%
-f_i^{l+1,j}\}|_{\gamma_{-}^i}=0.
\end{cases}
\eeq
We apply the estimate \eqref{j} to \eqref{jcon} to conclude that $f^{l+1,j}$ is Cauchy in $L^{\infty}$ for
$\delta\ll1$:
\begin{align}
\sup_{0\le t<\infty}\norm{\left\{f^{l+1,j+1}-f^{l+1,j}\right\}(t)}_{L^\infty(I)}&\lesssim\delta
\sup_{0\le t<\infty}\norm{\left\{f^{l+1,j}-f^{l+1,j-1}\right\}(t)}_{L^\infty(I)}\nonumber
\\&<\frac{1}{2}\sup_{0\le t<\infty}\norm{\left\{f^{l+1,j}-f^{l+1,j-1}\right\}(t)}_{L^\infty(I)}.
\end{align}

Now taking $\partial_{t}$ of (\ref{fnj}) yields
\beq
\begin{cases}
\lbrack\partial_{t}+\lambda_{li}^{\ast}\partial_{x}\rbrack \p_t f_{i}^{l+1,j+1}=-\p_t\lambda_{li}^{\ast}\partial_{x}  f_{i}^{l+1,j+1}+\{[\partial
_{t}+\lambda_{li}^{\ast}\partial_{x}]{T_{l}^\ast}^{-1}T_{l}^\ast \p_t f^{l+1,j}\}_{i}
\\\qquad\qquad\qquad\qquad\qquad\ +\{\p_t\{[\partial
_{t}+\lambda_{li}^{\ast}\partial_{x}]{T_{l}^\ast}^{-1}T_{l}^\ast \} f^{l+1,j}\}_{i},\\
\partial_{t}f_{i}^{l+1,j+1}|_{t=0}   =-\lambda_{li}^{\ast}|_{t=0}\partial_{x} f_{0,i}+\{[\partial
_{t}+\lambda_{li}^{\ast}\partial_{x}]{T_{l}^\ast}^{-1}T_{l}^\ast|_{t=0} f_0\}_{i}, \,\partial_{t}f_{i}^{l+1,j+1}|_{\gamma_{-}^i}=\partial_{t}b_{i}.
\end{cases}
\eeq
We also deduce that from (\ref{fnj}),
\beq
|\partial_{x}f_i^{l+1,j+1}|
=\frac{1}{|\lambda_{li}^{\ast}|}|\partial_{t}%
f_i^{l+1,j+1}-\{[\partial
_{t}+\lambda_{li}^{\ast}\partial_{x}]{T_{l}^\ast}^{-1}T_{l}^\ast f^{l+1,j}\}_{i}\}|\label{jx}
 \lesssim|\partial_{t}f^{l+1,j+1}|+\delta|f^{l+1,j}|.\nonumber
\eeq
By (\ref{infinity}) in Theorem \ref{scalar}, we deduce from \eqref{v*} and \eqref{fjes} that
\begin{align*}
&\sup_{0\le t<\infty}\norm{\partial_{t}f_i^{l+1,j+1}(t)}_{L^\infty(I)}+\norm{\partial_{t}%
f_i^{l+1,j+1}|_{\gamma}}_{L^\infty(\mathbb{R}_+)}
\\&\quad \ls \norm{\lambda_{li}^{\ast}|_{t=0}\partial_{x} f_{0,i}-\{[\partial
_{t}+\lambda_{li}^{\ast}\partial_{x}]{T_{l}^\ast}^{-1}T_{l}^\ast|_{t=0} f_0\}_{i}}_{L^\infty(I)}+\norm{\p_t b_i}_{L^{\infty}(\mathbb{R}_+)}
\\&\qquad
+\sup_{0\le t<\infty}\norm{\p_t\lambda_{li}^{\ast}\partial_{x}  f_{i}^{l+1,j+1}-\{[\partial
_{t}+\lambda_{li}^{\ast}\partial_{x}]{T_{l}^\ast}^{-1}T_{l}^\ast \p_t f^{l+1,j}\}_{i}}_{L^\infty(I)}
\\&\qquad
+\sup_{0\le t<\infty}\norm{\{\p_t\{[\partial
_{t}+\lambda_{li}^{\ast}\partial_{x}]{T_{l}^\ast}^{-1}T_{l}^\ast \} f^{l+1,j}\}_{i}(t)}_{L^\infty(I)}
\\&\quad  \lesssim \norm{f_{0}}_{W^{1,\infty}(I)}+\norm{\p_t b}_{L^{\infty}(\mathbb{R}_+)}+\delta\sup_{0\le t<\infty} \norm{\partial_{x}f^{l+1,j+1}(t)}_{L^\infty(I)}
\\&\qquad +\delta\sup_{0\le t<\infty} \norm{\partial_{t}f^{l+1,j}(t)}_{L^\infty(I)}  +l\sup_{0\le t<\infty} \norm{f^{l+1,j}(t)}_{L^\infty(I)}\\
&  \quad\lesssim l\left\{\norm{V_{0}}_{W^{1,\infty}(I)}+\norm{ b}_{W^{1,\infty}(\mathbb{R}_+)}\right\}
+\delta\sup_{0\le t<\infty} \norm{\partial_{t}f^{l+1,j+1}(t)}_{L^\infty(I)}
\\&  \qquad  +\delta\sup_{0\le t<\infty} \norm{\partial_{t}f^{l+1,j}(t)}_{L^\infty(I)}.
\end{align*}
We then have for $\delta\ll1,$
\beq
\sup_{0\le t<\infty} \norm{\partial_{t}f^{l+1,j+1}(t)}_{L^\infty(I)}+ \norm{\partial_{t}%
f^{l+1,j+1}|_{\gamma}}_{L^\infty(\mathbb{R}_+)}
\lesssim l\left\{\norm{V_{0}}_{W^{1,\infty}(I)}+\norm{ b}_{W^{1,\infty}(\mathbb{R}_+)} \right\}
\eeq
by induction on $j$. We thus conclude the proof by taking $j\rightarrow\infty$.
\end{proof}

We now ready to prove our main Theorem \ref{nbyn}.

\begin{proof}[\bf Proof of Theorem \ref{nbyn}]
Note that
$\norm{V^{l+1}}_{L^\infty}$ is bounded from (\ref{claim}).

\textit{Step 1.} \textit{Uniform }$\norm{\partial_{t}V^{l+1}}_{L^\infty
}$ \textit{estimate}. We take $\partial_{t}$ of (\ref{nsystem}) such that%
\begin{equation}
\lbrack\partial_{t}+ A_{l}^{\ast}\partial_{x}\rbrack\partial_{t}V^{l+1}%
=-\partial_{t}A_{l}^{\ast}{}\partial_{x}V^{l+1} ,\quad\partial_{t}V^{l+1}|_{t=0}=-A_{l}^{\ast}|_{t=0}\p_x({T_{l}^\ast}|_{t=0}T^{-1}_0V_{0}).
\end{equation}
We further define the good unknown for $\partial_{t}V^{l+1}$ as in (\ref{gi}):
\begin{equation}
g^{l+1}\equiv:T_{l}^{\ast-1}\partial_{t}V^{l+1}=\p_tf^{l+1}+T_{l}^{\ast-1}\p_tT_{l}^\ast f^{l+1}
\label{gg}%
\end{equation}
such that
\beq\label{g}
\begin{cases}
\lbrack\partial_{t}+\lambda_{li}^{\ast}\partial_{x}\rbrack g_{i}^{l+1}    =\{\lbrack\partial_{t}+\lambda_{li}^{\ast}\partial_{x}\rbrack T_{l}^{\ast-1}T_{l}^{\ast}g^{l+1}- T_{l}^{\ast-1} \partial_{t}A_{l}^{\ast}\partial_{x}V^{l+1}\}_i,
\\ g_{i}^{l+1}|_{t=0}= \{-T_{l}^{\ast-1}A_{l}^{\ast}|_{t=0} \p_x ({T_{l}^\ast}|_{t=0}T^{-1}_0V_{0}) \}_i,\quad
g_{i}^{l+1}|_{\gamma_{-}^{i}} =\p_t b_{i}+\{T_{l}^{\ast-1}\p_tT_{l}^\ast f^{l+1}\}_i.
\end{cases}
\eeq
From (\ref{gg}),
\beq\label{g_x}%
|g^{l+1}|\lesssim|\p_tV^{l+1}|\lesssim|g^{l+1}|.
\eeq
Also, from (\ref{nsystem}),
\begin{equation}
|\p_xV^{l+1}|\lesssim|\p_tV^{l+1}| .  \label{V_x}%
\end{equation}
We apply \eqref{infinity} of Theorem \ref{scalar} to \eqref{g} to have, by \eqref{claim},
\begin{align*}
&\sup_{0\le t<\infty} \norm{g_i^{l+1}(t)}_{L^\infty(I)}+\norm{g_i^{l+1}|_{\gamma}}_{L^\infty(\mathbb{R}_+)}
\\&\quad \ls  \norm{\{-T_{l}^{\ast-1}A_{l}^{\ast}|_{t=0} \p_x ({T_{l}^\ast}|_{t=0}T^{-1}_0V_{0}) \}_i}_{L^\infty(I)}+ \norm{\{\p_t b+T_{l}^{\ast-1}\p_tT_{l}^\ast f^{l+1}\}_i|_{\gamma_-^i}}_{L^\infty(\mathbb{R}_+)}
\\&\qquad+\sup_{0\le t<\infty}\norm{\{\lbrack\partial_{t}+\lambda_{li}^{\ast}\partial_{x}\rbrack T_{l}^{\ast-1}T_{l}^{\ast}g^{l+1}- T_{l}^{\ast-1} \partial_{t}A_{l}^{\ast}\partial_{x}V^{l+1}\}_i(t)}_{L^\infty(I)}
\\   &\quad\lesssim\norm{V_{0}}_{W^{1,\infty}(I)} +
\norm{\partial_{t}b}_{L^\infty(\mathbb{R}_+)} +\delta \norm{f^{l+1}|_{\gamma}}_{L^\infty(\mathbb{R}_+)}\\   &\qquad+\delta \sup_{0\le t<\infty} \norm{g^{l+1}(t)}_{L^\infty(I)}+\delta \sup_{0\le t<\infty} \norm{\partial_{x}V^{l+1}(t)}_{L^\infty(I)}  \\
& \quad  \lesssim \norm{V_{0}}_{W^{1,\infty}(I)}+\norm{ b}_{W^{1,\infty}(\mathbb{R}_+)}  +\delta \sup_{0\le t<\infty} \norm{g^{l+1}(t)}_{L^\infty(I)}.
\end{align*}
Then for \thinspace$\delta\ll1$,
\[
\sup_{0\le t<\infty}\norm{\p_tV^{l+1}(t)}_{L^\infty(I)}+ \norm{\p_tV^{l+1}|_{\gamma}}_{L^\infty(\mathbb{R}_+)} \lesssim \norm{V_{0}}_{W^{1,\infty}(I)}+\norm{ b}_{W^{1,\infty}(\mathbb{R}_+)}   .
\]
From (\ref{V_x}), we deduce from (\ref{claim})
\beq
\sup_{0\le t<\infty}\sum_{j=0}^1\norm{\p_t^j V^{l+1}(t)}_{W^{1-j,\infty}(I)}+ \norm{V^{l+1}|_{\gamma}}_{W^{1,\infty}(\mathbb{R}_+)}
\lesssim \norm{V_{0}}_{W^{1,\infty}(I)}+\norm{ b}_{W^{1,\infty}(\mathbb{R}_+)}   \le\delta  \label{nestimate}%
\eeq
by choosing $\eps_0 $ in \eqref{small}
sufficiently small. We therefore establish the uniform bound (\ref{inductionl}) by
induction on $l.$

\textit{Step 2. Take limit }$l\rightarrow\infty.$ We take difference
$\{V^{l+1}-V^{l}\}$ from (\ref{nsystem}) such that
\begin{equation}
\p_t\{V^{l+1}-V^{l}\}+ A_{l}^{\ast}\p_x\{V^{l+1}-V^{l}\}=[A_{l}^{\ast
}-A_{l-1}^{\ast}]\p_xV^{l}.
\end{equation}
We define
\beq
d^{l+1}:= T_{l}^{\ast-1}[V^{l+1}-V^{l}]=f^{l+1}-f^{l}-[T_{l}^{\ast-1}-T_{l-1}^{\ast-1}]V^{l},
\eeq
then
\beq\label{d}
\begin{cases}
\lbrack\partial_{t}+\lambda_{li}^{\ast}\partial_{x}]d_{i}^{l+1}    =\{[\partial
_{t}+\lambda_{i}^{l}\partial_{x}]T_{l}^{\ast-1}T_{l}^{\ast}d^{l+1}+T_{l}^{\ast-1}%
[A_{l}^{\ast}-A_{l-1}^{\ast}]\p_xV^{l}\}_{i},\\
d_{i}^{l+1}|_{t=0}=-\{[T_{l}^{\ast-1}-T_{l-1}^{\ast-1}]{T_{l-1}^\ast}|_{t=0}T^{-1}_0V_{0}\}_i,\quad d_{i}^{l+1}|_{\gamma_{-}^{i}}=-\{[T_{l}^{\ast-1}-T_{l-1}^{\ast-1}]V^{l}\}_i.
\end{cases}
\eeq
We then apply the $L^{\infty}$ estimate \eqref{infinity} to deduce
\begin{align*}
& \sup_{0\le t<\infty}\norm{d_{i}^{l+1}(t)}_{L^\infty(I)}+\norm{d_{i}^{l+1}|_{\gamma}}_{L^\infty(\mathbb{R}_+)}
\\&\quad\ls \norm{\{[T_{l}^{\ast-1}-T_{l-1}^{\ast-1}]{T_{l-1}^\ast}|_{t=0}T^{-1}_0V_{0}\}_i}_{L^\infty(I)}+ \sup_{0\le t<\infty} \norm{\{[T_{l}^{\ast-1}-T_{l-1}^{\ast-1}]V^{l}\}_i|_{\gamma_-^i}}_{L^\infty(\mathbb{R}_+)}
\\&\qquad+\sup_{0\le t<\infty}\norm{\{[\partial
_{t}+\lambda_{i}^{l}\partial_{x}]T_{l}^{\ast-1}T_{l}^{\ast}d^{l+1}+T_{l}^{\ast-1}%
[A_{l}^{\ast}-A_{l-1}^{\ast}]\p_xV^{l}\}_{i}(t)}_{L^\infty(I)}
\\&\quad\lesssim\delta \norm{d_0^{l}}_{L^\infty(I)}+\delta  \norm{d^{l}|_{\gamma}}_{L^\infty(\mathbb{R}_+)}+\delta  \sup_{0\le t<\infty} \norm{d^{l+1}(t)}_{L^\infty(I)}+\delta \sup_{0\le t<\infty} \norm{d^{l}(t)}_{L^\infty(I)}
\end{align*}
so that for $\delta
\ll1$,
\begin{align}
\sup_{0\le t<\infty}\norm{d^{l+1}(t)}_{L^\infty(I)}+\norm{d^{l+1}|_{\gamma}}_{L^\infty(\mathbb{R}_+)}&\lesssim \delta\left\{\sup_{0\le t<\infty}\norm{d^{l}(t)}_{L^\infty(I)}+ \norm{d^{l}|_{\gamma}}_{L^\infty(\mathbb{R}_+)}\right\}
\nonumber\\&<\frac{1}{2}\left\{\sup_{0\le t<\infty}\norm{d^{l}(t)}_{L^\infty(I)}+ \norm{d^{l}|_{\gamma}}_{L^\infty(\mathbb{R}_+)}\right\}.
\end{align}
We then deduce $V^{l}$ is Cauchy in $L^{\infty}$.

We take $l\rightarrow\infty$ in \eqref{nsystem} and $V^{l}\rightarrow V$ so that $V$ solves
(\ref{system}) with the desired estimate (\ref{vdecay}) from (\ref{nestimate})
as $l\rightarrow\infty.$ This concludes the existence part of Theorem \ref{nbyn}.

\textit{Step 3. Uniqueness.} We assume that there are two solutions
$V_{1}$ and $V_{2}$ to \eqref{system} with the same data and the estimate
\eqref{vdecay}. Taking difference yields%
\beq
\partial_{t}\{V_{1}-V_{2}\}+A(V_{1})\partial_{x}\{V_{1}-V_{2}\}+[A(V_{1}%
)-A(V_{2})]\partial_{x}V_{2}=0.
\eeq
We define
\beq
d:=T_{1}^{-1}[V_{1}-V_{2}]=f_{1}-f_{2}+[T_{1}^{-1}-T_{2}^{-1}]V_{2},%
\eeq
then
\beq
\begin{cases}
[\partial_{t}+\Lambda(V_{1})\partial_{x}]d    =[\partial_{t}+\Lambda(V_{1})\partial_{x}]T_{1}^{-1}T_{1}d-[A(V_{1})-A(V_{2})]\partial
_{x}V_{2},\\
d|_{t=0}   =0,\quad d_{i}|_{\gamma_{-}^{i}}=[T_{1}^{-1}%
-T_{2}^{-1}]V_{2}.
\end{cases}
\eeq
From the $L^{\infty}$ estimate \eqref{infinity}, \eqref{vdecay} and the smallness of \eqref{small}, we conclude the
uniqueness:
\[\sup_{0\le t<\infty}\norm{d(t)}_{L^\infty(I)}+\norm{d|_{\gamma}}_{L^\infty(\mathbb{R}_+)}\lesssim \eps_0 \left\{\sup_{0\le t<\infty}\norm{d(t)}_{L^\infty(I)}+\norm{d|_{\gamma}}_{L^\infty(\mathbb{R}_+)}\right\}\implies d\equiv0.\]
The proof of Theorem \ref{nbyn} is thus completed.
\end{proof}

\section{3D Incompressible Euler System}

\subsection{Linear Transport Equation}
Assume $u\in W^{1,\infty}(\mathbb{R}_+\times \Omega)$ with $
u_1 >0$, $\dive u=0$ in $\Omega$ and $u\cdot \nu=0$ on $\gamma_0$.
We consider the linear transport inflow problem:
\begin{align}
 & \partial_t f+u\cdot\nabla f     =h, \label{transport1}\\
 & f|_{t=0}   =f_{0} ,\quad f|_{\gamma_-}=b.\label{initial1}
\end{align}

Fix $t\geq0$ and $x\in \bar \Omega,$ we define the characteristic
curve  $X(\tau;t,x)$ passing through $(t,x)$:
\begin{equation}
\frac{d}{d\tau}X(\tau;t,x)=u(\tau,X(\tau;t,x)),\quad
X(t;t,x)=x. \label{char2}
\end{equation}
\begin{lemma}\label{inlem}
Assume $u\cdot \nu=0 \text{ on }\gamma_0$, then
\[
\{X(\tau;t,x),\tau \ge 0\} \subseteq \gamma_0, \text{ for }t\ge0,\ x\in \gamma_0.
\]
\end{lemma}
\begin{proof}
For $t\ge 0$ and $x\in\gamma_0$, we define the sub-characteristic curve  $\tilde X(\tau;t,x)$ ($\subseteq \gamma_0$) passing through $(t,x)$: denoting $u_\parallel=u-u\cdot \nu \nu,$
\[
\frac{d}{d\tau}\tilde X(\tau;t,x)=u_\parallel(\tau,\tilde X(\tau;t,x)),\quad
\tilde X(t;t,x)=x.
\]
Since $u\cdot \nu=0$ on $ \gamma_0$, it follows that
\[
\frac{d}{d\tau}\tilde X(\tau;t,x)=u(\tau,\tilde X(\tau;t,x)),\quad
\tilde X(t;t,x)=x.
\]
By the uniqueness of ODE \eqref{char2}, $X(\tau;t,x)=\tilde X(\tau;t,x) \subseteq \gamma_0$.
\end{proof}

For $t\ge0$ and $ x\in \bar\Omega$,   we define the backward exit-time for $(t,x)$ as
\[
t_{b}^{-}(t,x)=\inf_{s\leq t}\{ X(s;t,x)\in \bar\Omega\}.
\]
Since $u_1\ge c_1>0$, it follows as 1D case that $t_{b}^{-}(t,x)$ is well-defined for $t\ge0$, $x\in \bar\Omega$. Note that  $0\leq t_{b}^{-}(t,x)\leq t,$   and  $t_{b}^{-}(t,x)=t$ if and only if   $(t,x)\in \{t\ge 0, x\in \bar{\gamma}_-\}\cup\{t=0,x\in \bar\Omega \}$.
We also define the
backward exit point $x_{b}^{-}(t,x)=X(t_{b}^{-}(t,x);t,x)$ such that  $(t_{b}^{-},x_{b}^{-})\in \{t\ge 0, x\in \bar{\gamma}_-\}\cup\{t=0,x\in \bar\Omega \}$.
Moreover, we can split $\{t\ge0, x\in \bar\Omega\}$ into disjoint sets:
\begin{align}
Q_+ &:=\{(t,x)     :t_{b}^{-}(t,x)=0,\  x_{b}^{-}
(t,x)\in \bar\Omega\backslash \bar{\gamma}_-\},\nonumber
\\Q_- &:=\{(t,x)   :t_{b}^{-}(t,x)>0,\  x_{b}^{-}%
(t,x) \in  \bar{\gamma}_-\}
,\label{split1}\\
\Gamma &:=\{(t,x)    :t_{b}^{-}(t,x)=0,\  x_{b}^{-}%
(t,x) \in \bar{\gamma}_-\} \equiv\{(\tau,X(\tau;0,x))\mid x\in \bar{\gamma}_-,\tau\ge 0\},\nonumber
\end{align}
and it holds that $t_{b}^{-},x_{b}^{-} \in
W^{1,\infty}$, uniformly up to the $W^{1,\infty}$ hypersurface $\Gamma$.

For $t\ge 0$, $x\in \bar \Omega$, we define
\begin{align}\label{expression1}
f(t,x)  &  =
\begin{cases}
\dis f_{0}(X(0;t,x))+\int_{0}^{t}h(\tau;X(\tau;t,x))d\tau,&\text{if }
(t,x)\in Q_+,\\
\dis  b(t_{b}^{-}(t,x), x_{b}^{-}(t,x) )+\int_{t_{b}^{-}(t,x)}^{t}h(\tau;X(\tau;t,x))d\tau,&\text{if }
(t,x)\in Q_-.
\end{cases}
\end{align}
Then from \eqref{split1}, $f(t,x)$ is well-defined mild solution to \eqref{transport1} and \eqref{initial1}
except along $\Gamma$, while the trace $f|_{\gamma_{+}}$ is well-defined except the intersection point of $\Gamma$ with $[0,\infty)\times\gamma_+$.

\begin{theorem}\label{scalar2}\bigskip
Let $3<p<\infty$. Assume $0<c_1\le u_1\le c_2<\infty$. For fixed $\alpha\ge 1$, define the spatial weight function:
\begin{equation}\label{weight2}
w=w(x_1):=e^{-\alpha  x_1}.
\end{equation}
(1) Assume $f_0\in L^p(\Omega),\ b\in L^\infty(0,\infty; L^p(\gamma_-)),\ h\in L^\infty(0,\infty; L^p(\Omega)).$
Then there exists a unique solution $f$ to \eqref{transport1} and \eqref{initial1} such that
\begin{align}\label{esf}
  \sup_{0\le t<\infty}\norm{ w  f(t)}_{L^p(\Omega)}^p     \ls   \norm{ w  f_0}_{L^p(\Omega)}^p+ \frac{c_2}{  c_1}    \norm{ w  b }_{L^\infty(0,\infty;L^p(\gamma_-))}^p + (c_1\alpha)^{-p}  \norm{ w  h }_{L^\infty(0,\infty;L^p(\Omega))}^p .
\end{align}
(2) Assume $f_0\in W^{1,p}(\Omega), \p_t^j b \in L^\infty(0,\infty; W^{1-j,p}(\gamma_-))  ,  \p_t^j h \in L^\infty(0,\infty; W^{1-j,p}(\Omega)), j=0,1,$ and the compatibility condition $
f_{0}|_{\gamma_{-}}=b(0)$. If
\beq\label{alcon}
 \norm{   \bar \p u  }_{L^\infty((0,\infty)\times\Omega)}^p  \left(1+ c_1^{-p}+ c_1^{-p}\norm{(u_2 ,u_3)  }_{L^\infty((0,\infty)\times\Omega)}^p\right)\le \delta (c_1\alpha)^{p},
\eeq
where $\bar \p =(\p_t,\p_2,\p_3)$, then
\begin{align}\label{esf2}
 &\sup_{0\le t<\infty}\norm{ w \bar \p f  (t)}_{L^p(\Omega)}^p
  \ls   \norm{ w \bar \p  f(0)}_{L^p(\Omega)}^p+ \frac{c_2}{  c_1}     \norm{ w \bar \p  b }_{L^\infty(0,\infty;L^p(\gamma_-))}^p
\nonumber\\&\qquad + (c_1\alpha)^{-p}c_1^{-p}  \norm{   \bar \p u  }_{L^\infty((0,\infty)\times\Omega)}^p   \norm{
w h}_{L^\infty(0,\infty;L^p(\Omega))}^p+ (c_1\alpha)^{-p}  \norm{ w \bar \p  h }_{L^\infty(0,\infty;L^p(\Omega))}^p .
\end{align}
(3) If further  $f_0=0$ on $\gamma_0$, $b=0$ on $\gamma_-\cap\gamma_0$ and $h=0$ on $\gamma_0$, then $f=0$ on $\gamma_0$.
\end{theorem}
\begin{proof}
Owing to $u\cdot \nu=0$ on $\gamma_0$, the Green identity in the transport theory \cite{GKTT} is not directly valid for \eqref{transport1}. We shall utilize a non-characteristic approximation; for sufficiently small $\varepsilon>0$, we consider the approximating transport inflow problem:
\begin{align}
 & \partial_t f^\eps+ u^\varepsilon \cdot\nabla f^\eps     =h,\quad u^\varepsilon:=u+\varepsilon (0,x_2,x_3), \label{transport1ep}\\
 &f^\eps|_{t=0}   =f_0,\quad f^\eps|_{\gamma_-}=b.\label{initial1ep}
\end{align}
Here, we note that $
u^\varepsilon\cdot \nu=\varepsilon>0$ on $\gamma_0$; indeed, for the circular pipe $\nu=(0,x_2,x_3)$ on $\gamma_0$ so that $\varepsilon (0,x_2,x_3)\cdot \nu\equiv \eps |\nu|^2=\eps$, while for the square pipe, for instance, $\nu=\pm e_2$ on $\{x_2=\pm 1\}\cap\gamma_0$ so that $\varepsilon (0,x_2,x_3)\cdot \nu\equiv\pm \eps x_2=\eps$.

Fix $t\geq0$ and $x\in \bar \Omega,$ we define the characteristic
curve  $X^\varepsilon(\tau;t,x)$ passing through $(t,x)$:
\begin{equation}
\frac{d}{d\tau}X^\varepsilon(\tau;t,x)=u^\varepsilon(\tau,X(\tau;t,x)),\quad
X^\varepsilon(t;t,x)=x. \label{char2ep}
\end{equation}
We define the backward exit-time for $(t,x)$ as
\[
t_{b}^{\varepsilon,-}(t,x)=\inf_{s\leq t}\{ X^\varepsilon(s;t,x)\in \bar\Omega\}
\]
and the backward exit point $x_{b}^{\varepsilon,-}(t,x)=X(t_{b}^{\varepsilon,-};t,x)$ such that  $(t_{b}^{\varepsilon,-},x_{b}^{\varepsilon,-})\in\{t\ge0,x\in\gamma_{-}\}$ $\cup\{t=0,x\in \bar\Omega\} $.
Moreover, we can split $\{t\ge0, x\in \bar\Omega\}$ into disjoint sets:
\begin{align}
Q_+^\varepsilon &:=\{(t,x)     :t_{b}^{\varepsilon,-}(t,x)=0,\  x_{b}^{\varepsilon,-}
(t,x)\in \bar\Omega\backslash \gamma_{-}\},\nonumber
\\Q_-^\varepsilon &:=\{(t,x)   :t_{b}^{\varepsilon,-}(t,x)>0,\ x_{b}^{\varepsilon,-}%
(t,x) \in \gamma_{-}\},
\label{split2e}\\
\Gamma^\varepsilon &:=\{(t,x)    :t_{b}^{\varepsilon,-}(t,x)=0,\  x_{b}^{\varepsilon,-}%
(t,x) \in \gamma_{-}\} \equiv\{(\tau,X^\varepsilon(\tau;0,x))\mid x\in \gamma_{-},\tau\ge 0\},\nonumber
\end{align}
and $t_{b}^{\varepsilon,-},x_{b}^{\varepsilon,-} \in
W^{1,\infty}$, uniformly up to the $W^{1,\infty}$ hypersurface $\Gamma^\varepsilon$.

For $t\ge 0$, $x\in \bar \Omega$, we define
\begin{align}\label{expression2}
f^\eps(t,x)  &  =
\begin{cases}
\dis f_{0}(X^\eps(0;t,x))+\int_{0}^{t}h(\tau;X^\eps(\tau;t,x))d\tau,&\text{if }
(t,x)\in Q_+^\eps,\\
\dis  b\left(t_{b}^{\eps,-}(t,x), x_{b}^{\eps,-}(t,x)\right)+\int_{t_{b}^{\eps,-}(t,x)}^{t}h(\tau;X^\eps(\tau;t,x))d\tau,&\text{if }
(t,x)\in Q_-^\eps.
\end{cases}
\end{align}
From \eqref{split2e}, it follows that $f^\eps(t,x)$ is well-defined mild solution to \eqref{transport1ep} and \eqref{initial1ep}
except along  $\Gamma^\eps$, while the
trace $f^\eps|_{\gamma_+\cup\gamma_0}$ is well-defined except the intersection point of $\Gamma^\eps$ with $[0,\infty)\times\{\gamma_+\cup\gamma_0\}$.

$\bullet$ {\it Proof of Part (1).} Recall the weight function $w(x_1)$ from \eqref{weight2} with $w'(x_1)= -\alpha w(x_1)$.
Then we have from \eqref{transport1ep}, away from $\Gamma^\eps$, for $3<p<\infty,$
\begin{align}\label{fpeq}
&\frac{1}{p} \partial_t \left\{e^{\alpha c_1 t}w^p|f^\eps|^p\right\}+\frac{1}{p}  u\cdot\nabla \left\{e^{\alpha c_1 t}w^p|f^\eps|^p \right\}+ \frac{1}{p}  (p u_1-c_1) \alpha \left\{e^{\alpha c_1 t}w^p|f^\eps|^p\right\}
\nonumber\\&\quad= e^{\alpha c_1 t}w^p |f^\eps|^{p-2}f^\eps h.
\end{align}
By the Green identity \cite{GKTT}, recalling $u^\varepsilon\cdot \nu=\eps$ on $\gamma_0$ and ${\rm div}u^\eps=2\eps$ in $\Omega$,
\begin{align}
&\frac{1}{p}\int_{\Omega}   e^{\alpha c_1 t} w^p |f^\eps(t)|^{p }    +\frac{1}{p}    \int_0^t \int_{\Omega} \left[ (p u_1-c_1) \alpha-2\eps\right] e^{\alpha c_1 \tau}  w^{p}|f^\eps|^p \nonumber
 \\&\quad+ \frac{1}{p}  \int_0^t \int_{\gamma_+}   u_1   e^{\alpha c_1 \tau}w^p|f^\eps|^p+ \frac{1}{p}  \int_0^t \int_{\gamma_0}   \varepsilon   e^{\alpha c_1 \tau}w^p|f^\eps|^p
\nonumber\\&\quad = \frac{1}{p}\int_{\Omega}   w^p  |f^\eps(0)|^{p }  +\frac{1}{p}    \int_0^t \int_{\gamma_-}   u_1  e^{\alpha c_1 \tau}  w^{p}|f^\eps|^p +\int_0^t \int_{\Omega} e^{\alpha c_1 \tau}w^p |f^\eps|^{p-2}f h
\nonumber\\&\quad = \frac{1}{p}\int_{\Omega}   w^p  |f_0|^{p }+ \frac{1}{p}  \int_0^t \int_{\gamma_-}  u_1  e^{\alpha c_1 \tau}w^p|b|^p +\int_0^t \int_{\Omega}e^{\alpha c_1 \tau}w^p |f^\eps|^{p-2}f^\eps h .\nonumber
\end{align}
By Young's inequality and taking $\eps$ sufficiently small, we have
\begin{align}
& e^{\alpha c_1 t}\norm{ w  f^\eps(t)}_{L^p(\Omega)}^p     +  c_1 \alpha \int_0^t e^{\alpha c_1 \tau}\norm{ w  f^\eps(\tau)}_{L^p(\Omega)}^p
\nonumber\\&\quad +   c_1\int_0^t e^{\alpha c_1 \tau}\norm{ w  f^\eps(\tau)}_{L^p(\gamma_+)}^p+\varepsilon\int_0^t      e^{\alpha c_1 \tau}\norm{ w  f^\eps(\tau)}_{L^p(\gamma_0)}^p
\nonumber\\&\quad \ls   \norm{ w  f_0}_{L^p(\Omega)}^p+ c_2  \int_0^t      e^{\alpha c_1 \tau}\norm{ w  b(\tau)}_{L^p(\gamma_-)}^p + (c_1\alpha)^{-(p-1)}\int_0^t  e^{\alpha c_1 \tau}\norm{ w  h(\tau)}_{L^p(\Omega)}^p.\nonumber
\end{align}
Dividing $e^{\alpha c_1 t}$, we obtain
\begin{align}\label{esfe}
& \norm{ w  f^\eps(t)}_{L^p(\Omega)}^p     +    c_1 \alpha\int_0^t e^{-\alpha c_1 (t-\tau)}\norm{ w  f^\eps(\tau)}_{L^p(\Omega)}^p
  \nonumber\\&\quad+   c_1\int_0^t e^{-\alpha c_1 (t-\tau)}\norm{ w  f^\eps(\tau)}_{L^p(\gamma_+)}^p+\varepsilon\int_0^t      e^{-\alpha c_1 (t-\tau)}\norm{ w  f^\eps(\tau)}_{L^p(\gamma_0)}^p
\nonumber\\&\quad \ls
 \norm{ w  f_0}_{L^p(\Omega)}^p+c_2  \int_0^t      e^{-\alpha c_1 (t-\tau)}\norm{ w  b(\tau)}_{L^p(\gamma_-)}^p + (c_1\alpha)^{-(p-1)}\int_0^t  e^{-\alpha c_1 (t-\tau)}\norm{ w  h(\tau)}_{L^p(\Omega)}^p
\nonumber\\&\quad \ls   \norm{ w  f_0}_{L^p(\Omega)}^p+  \frac{c_2}{\alpha c_1}    \norm{ w  b }_{L^\infty(0,\infty;L^p(\gamma_-))}^p + (c_1\alpha)^{-p}  \norm{ w  h }_{L^\infty(0,\infty;L^p(\Omega))}^p .
\end{align}
Based on the uniform estimates \eqref{esfe}, we can pass to the limit as $\varepsilon\rightarrow0$ in \eqref{transport1ep} and \eqref{initial1ep} to get a weak solution $f$  to \eqref{transport1} and \eqref{initial1} satisfying the estimate \eqref{esf}. By the uniqueness as in $1D$, $f$ is just the one given by \eqref{expression1}.

$\bullet$ {\it Proof of Part (2).} We utilize the non-characteristic approximation:
\begin{align}
 & \partial_t f^\eps+ u^\varepsilon \cdot\nabla f^\eps     =h^\eps:=\eta_\eps\ast \tilde h,\quad u^\varepsilon:=u+\varepsilon (0,x_2,x_3), \label{transport1ep2}\\
 &f^\eps|_{t=0}   =f^\eps_0:=\eta_\eps\ast \tilde f_{0} ,\quad f^\eps|_{\gamma_{-}}=b^\eps:= \eta_\eps\ast \tilde b-\eta_\eps\ast \tilde b(0)+f^\eps_0,\label{initial1ep2}
\end{align}
where $\tilde h$, $\tilde f_{0}$ and $\tilde b$ are the (spatial) Sobolev extensions of $h$, $f_0$, $b$ onto $\mathbb{R}^3$ or $\mathbb{R}^2$, respectively, and $\eta$ is the standard mollifier. Note that it follows from the compatibility condition $
f_{0}|_{\gamma_{-}}=b(0)$ that $-\eta_\eps\ast \tilde b(0)+f^\eps_0\rightarrow0$ on $\gamma_-$ as $\eps\rightarrow0$.

Recall that the solution $f^\eps$ to \eqref{transport1ep2} and \eqref{initial1ep2} is defined by \eqref{expression2} with $f_0,b,h$ replaced by $f_0^\eps, b^\eps,h^\eps$ accordingly. Then $f^\eps\in W^{1,\infty}([0,\infty)\times \bar\Omega)$ except  $\Gamma^\eps $ but up to $\Gamma^\eps$ from both sides $Q_\pm^\eps$. Now we show $f^\eps\in W^{1,\infty}([0,\infty)\times \bar\Omega)$  by proving the continuity along $\Gamma^\eps$ by noting the compatibility condition from our construction \eqref{initial1ep2}:
\begin{equation}\label{compat2}
b^\eps(0,x)=f_0^\eps(x),\quad x\in \gamma_{-}.
\end{equation}
Indeed, for $(t,x)\in  \Gamma^\eps $, that is,  $t_{b}^{\eps,-}(t,x)=0$ and $ x_{b}^{\eps,-}%
(t,x)=X^\eps(0;t,x)\in \gamma_{-} $, we define $f^\eps(t,x)$ via either the limit by approaching from $Q_+^\eps $:
\begin{align}
f^\eps(t,x)  &  =f_{0}^\eps(X^\eps(0;t,x))+\int_{0}^{t}h^\eps(\tau;X^\eps(\tau;t,x))d\tau
\nonumber
\\&=f_{0}^\eps(x_{b}^{\eps,-}%
(t,x))+\int_{0}^{t}h^\eps(\tau;X^\eps(\tau;t,x))d\tau\nonumber
\end{align}
or via the limit by approaching from $Q_-^\eps$:
\begin{align}
f^\eps(t,x)  &  =b(t_{b}^{\eps,-}(t,x), x_{b}^{\eps,-}(t,x) )+\int_{t_{b}^{\eps,-}(t,x)}^{t}h^\eps(\tau;X^\eps(\tau;t,x))d\tau\nonumber
\\&=b^\eps(0, x_{b}^{\eps,-}(t,x) )+\int_{0}^{t}h^\eps(\tau;X^\eps(\tau;t,x))d\tau.\nonumber
\end{align}
Note that thanks to \eqref{compat2}, $b^\eps(0, x_{b}^{\eps,-}(t,x) )=f_{0}^\eps(x_{b}^{\eps,-}%
(t,x))$ as $x_{b}^{\eps,-}(t,x)\in\gamma_{-}$, and hence these two are equal. So now $f^\eps$ is continuous in $t\ge0$, $x\in  \bar \Omega$ and hence $f^\eps\in W^{1,\infty}([0,\infty)\times \bar\Omega)$, and $f^\eps(t,x)$ is well-defined by \eqref{expression2} for  all $(t,x)\in[0,\infty)\times\bar\Omega$.

We next apply $\bar \p =\p_t,\p_2,\p_3$  to \eqref{transport1ep}, preserving the boundary conditions on $\gamma_{-}$, to obtain
\[\p_t \bar \p f^\eps +u^\eps \cdot\nabla \bar \p  f^\eps  =-\bar \p u^\eps\cdot\nabla   f^\eps+\bar \p  h^\eps .\]
Note that
\begin{align}
 &\bar \p f^\eps(t,x)=\bar \p b^\eps(t,x),\ (t,x)\in\gamma_-, \nonumber\\
 & \partial_t f^\eps(0,x)=-u^\eps(0,x)\cdot\nabla f_0^\eps(x)  +h^\eps(0,x),\ \p_{i} f^\eps(0,x)   =\p_{i} f_{0}^\eps(x)(i=2,3),\ x\in\Omega.\nonumber
\end{align}
We apply the estimate \eqref{esfe} to $\bar \p f^\eps$ to have
\begin{align}
& \norm{ w \bar \p f^\eps (t)}_{L^p(\Omega)}^p     + c_1 \alpha\int_0^t e^{-\alpha c_1 (t-\tau)}\norm{ w  \bar \p f^\eps(\tau)}_{L^p(\Omega)}^p
\nonumber\\&\quad +   c_1\int_0^t e^{-\alpha c_1 (t-\tau)}\norm{ w  \bar \p f^\eps(\tau)}_{L^p(\gamma_+)}^p +  \eps\int_0^t e^{-\alpha c_1 (t-\tau)}\norm{ w  \bar \p f^\eps(\tau)}_{L^p(\gamma_0)}^p
\nonumber\\&\quad \ls   \norm{ w \bar \p  f^\eps(0)}_{L^p(\Omega)}^p+ \frac{c_2}{\alpha c_1}    \norm{ w \bar \p  b^\eps }_{L^\infty(0,\infty;L^p(\gamma_-))}^p
\nonumber\\&\qquad + (c_1\alpha)^{-p}  \norm{   \bar \p u^\eps }_{L^\infty((0,\infty)\times\Omega)}^p   \norm{ w  \nabla   f^\eps }_{L^p(\Omega)}^p + (c_1\alpha)^{-p}   \norm{ w \bar \p  h^\eps }_{L^\infty(0,\infty;L^p(\Omega))}^p  .\nonumber
\end{align}
We also solve for $u_1>0$:
\[ \partial_{1}f^\eps      =\frac{1}{u_1}\left\{ -\partial_t f^\eps- u_2^\eps\p_2 f^\eps- u_3^\eps\p_3 f^\eps+ h^\eps\right\} \]
so that
\begin{align}
  \norm{
w\partial_{1}f^\eps}_{L^p(\Omega)}^p    \le c_1^{-p}  \left(1+\norm{(u_2^\eps,u_3^\eps) }_{L^\infty(\Omega)}^p\right)\norm{
  w\bar \p f^\eps
}_{L^p(\Omega)}^p+c_1^{-p}  \norm{
w h^\eps}_{L^p(\Omega)}^p .
\end{align}
Hence,
\begin{align}
  &  \norm{ w \bar \p f^\eps (t)}_{L^p(\Omega)}^p
 \ls   \norm{ w \bar \p  f(0)}_{L^p(\Omega)}^p+ \frac{c_2}{\alpha c_1}    \sup_{[0,t]}\norm{ w \bar \p  b }_{L^p(\gamma_-)}^p
\nonumber\\&\quad + (c_1\alpha)^{-p}\norm{   \bar \p u^\eps  }_{L^\infty((0,\infty)\times\Omega)}^p  \left(1+ c_1^{-p}+ c_1^{-p}\norm{(u_2^\eps ,u_3^\eps)  }_{L^\infty((0,\infty)\times\Omega)}^p\right) \norm{
  w\bar \p f^\eps
}_{L^p(\Omega)}^p
\nonumber\\&\quad + (c_1\alpha)^{-p}  \norm{   \bar \p u^\eps }_{L^\infty((0,\infty)\times\Omega)}^p  c_1^{-p} \norm{
w h^\eps}_{L^\infty(0,\infty;L^p(\Omega))}^p + (c_1\alpha)^{-p}  \norm{ w \bar \p  h^\eps }_{L^\infty(0,\infty;L^p(\Omega))}^p.\nonumber
\end{align}
Therefore,  by the assumption \eqref{alcon} and taking $\eps$ sufficiently small, we deduce
\begin{align}\label{esf2ee}
 &\norm{ w \bar \p f^\eps (t)}_{L^p(\Omega)}^p
  \ls   \norm{ w \bar \p  f(0)}_{L^p(\Omega)}^p+ \frac{c_2}{ c_1}    \sup_{[0,t]}\norm{ w \bar \p  b }_{L^p(\gamma_-)}^p
\nonumber\\&\quad  + (c_1\alpha)^{-p} \norm{   \bar \p u^\eps }_{L^\infty((0,\infty)\times\Omega)}^p c_1^{-p}   \norm{
w h^\eps}_{L^\infty(0,\infty;L^p(\Omega))}^p + (c_1\alpha)^{-p}  \norm{ w \bar \p  h^\eps }_{L^\infty(0,\infty;L^p(\Omega))}^p.
\end{align}
Based on the uniform estimates \eqref{esf2ee}, for $\eps$ smaller if necessary, we can take $\varepsilon\rightarrow0$ to find that $f$ is a strong solution to \eqref{transport1} and \eqref{initial1} and satisfies the estimate \eqref{esf2}.

$\bullet$ {\it Proof of Part (3).} Recall that $f$ is well-defined by \eqref{expression1}  except along the hypersurface $\Gamma$ and the lateral boundary $[0,\infty)\times\gamma_0$. Since $3<p<\infty$, by the Aubin--Lions lemma and Sobolev's embedding, $f\in C([0,\infty)\times\bar\Omega)$, and therefore $f$ is well-defined by \eqref{expression1} for all $(t,x)\in[0,\infty)\times\bar\Omega$.

Now take $(t,x)\in \{t\ge0,x\in\gamma_0\}$.
If $(t,x)\in   Q_-$, by the invariant of Lemma \ref{inlem}, $(t,x)\in \{(t,x)   :t_{b}^{-}(t,x)>0,\ x_{b}^{-}%
(t,x) \in  \gamma_-\cap\gamma_0\}$. Since $b=0$ on $\gamma_-\cap\gamma_0$ and $h=0$ on $\gamma_0$, by \eqref{expression1}, we have
\[
f(t,x) =b(t_{b}^{-}(t,x), x_{b}^{-}(t,x) )+\int_{t_{b}^{-}(t,x)}^{t}h(\tau;X(\tau;t,x))d\tau=0
\]
due to that $x_{b}^{-}%
(t,x) \in \gamma_-\cap\gamma_0$ and $X(\tau;t,x)\in\gamma_0$ by Lemma \ref{inlem}.

If $(t,x)\in  Q_+$, by the invariant of Lemma \ref{inlem}, $(t,x)\in \{(t,x)   :t_{b}^{-}(t,x)=0,\ x_{b}^{-}%
(t,x) \in  \gamma_0\}$. Since $f_0=0$ on $\gamma_0$ and $h=0$ on $\gamma_0$, by \eqref{expression1}, we have
\[
f(t,x)
= f_{0}(x_{b}^{-}%
(t,x) )+\int_{0}^{t}h(\tau;X(\tau;t,x))d\tau=0
\]
due to that $x_{b}^{-}%
(t,x) \in  \gamma_0$ and $X(\tau;t,x)\in\gamma_0$ by Lemma \ref{inlem}.

From both and $f\in C([0,\infty)\times\bar\Omega)$, we conclude that $f=0$ on $\gamma_0$.
\end{proof}

\subsection{Vorticity Equations}
Now given $v$, we consider the vorticity inflow problem:
\begin{equation} \label{omegan}
\begin{cases}
\p_t\omega +(u_s+v )\cdot\nabla \omega =- v  \cdot\nabla   \omega_s+  \omega_s \cdot \nabla  v +  \omega  \cdot \nabla  (u_s +    v )\quad\text{in }\Omega,\\
\omega |_{t=0}=\omega_0,\quad \omega |_{\gamma_-}  =\omega_b^-.
\end{cases}
\end{equation}
\begin{theorem}\label{vorth}
Let $3<p<\infty$.
Suppose $v$ is given such that
\begin{equation}\label{bbav}
{\rm div}v =0 \text{ in }\Omega,\  v \cdot\nu=0\text{ on }\gamma_0,\  U+v_1>0\text{ in }\bar\Omega.
\end{equation}
(1) Assume the compatibility conditions: $
\omega_0|_{\gamma_-}=\omega_b^-(0)$.
There exists $\delta>0$ such that if
\beq\label{smallness}
\norm{  \nabla  U
}_{W^{2,p}(\Omega)} < \delta   \frac{\min_{\bar\Omega} U^2}{1+ \min_{\bar\Omega} U}
\text{ and } \sup_{0\le t<\infty}
\sum_{j=0}^1\norm{
\p_t^j  v(t)
}_{W^{2-j,p}(\Omega)}< \delta  \min_{\bar\Omega} U,
\eeq
then there is a unique solution $\omega$ to \eqref{omegan} satisfying
\begin{align}\label{omegaes}
& \sup_{0\le t<\infty}\sum_{j=0}^1\norm{
 \p_t^j \omega(t)
}_{W^{1-j,p}(\Omega)}^p
 \nonumber\\&\quad  \ls   \norm{
 \omega_0
}_{W^{1,p}(\Omega)}^p+    \sum_{j=0}^1\norm{
 \p_t^j \omega_b^-
}_{L^\infty(0,\infty;W^{1-j,p}(\gamma_-))}^p
+    \delta  \sup_{0\le t<\infty}\sum_{j=0}^1\norm{
\p_t^j  v(t)
}_{W^{2-j,p}(\Omega)}^p      .
\end{align}
(2) Moreover, if  $v_b^-:=v_1$ on $\gamma_-$ and $\omega_b^-$ satisfy \eqref{bbcondition1} and $\dive \omega_0=0$ in $\Omega$, then
\beq\label{omegazero}
\dive \omega=0\text{ in }\Omega;
\eeq
(3) For $D=(-1,1)^2$, if
$
\omega_s\times \nu=0$ on $\gamma_0$, $\omega_0\times \nu=0$ on $\gamma_0$ and $\omega_b^- \times \nu =0$ on $\gamma_-\cap\gamma_0$, then
\beq\label{oze2}
\omega\times \nu=0\text{ on }\gamma_0;
\eeq
For $D=B(0,1)$, if $
\omega_s=0$ on $\gamma_0$, $\omega_0=0$ on $\gamma_0$ and $\omega_b^- =0$ on $\gamma_-\cap\gamma_0$, then
\beq\label{oze22}
\omega =0\text{ on }\gamma_0.
\eeq
\end{theorem}
\begin{proof}
We construct the  solution $\omega$ to \eqref{omegan} by the iteration.

$\bullet$ {\it Proof of Part (1).}
We start with $\omega^{1}=0$, and  define iteratively $\omega^{l+1}$ as the solution to
\begin{equation}\label{omj}
\begin{cases}
\p_t\omega^{l+1} +(u_s+v )\cdot\nabla \omega^{l+1} =- v  \cdot\nabla   \omega_s+  \omega_s \cdot \nabla  v +  \omega^{l}  \cdot \nabla ( u_s +    v )\quad\text{in }\Omega,\\
\omega^{l+1} |_{t=0}=\omega_0,\quad\omega^{l+1}|_{\gamma_-} =\omega_b^- ,
\end{cases}
\end{equation}
which is guaranteed by Theorem \ref{scalar2}. Moreover, if we have assumed \eqref{smallness} such that \eqref{alcon} is satisfied, then by \eqref{esf2}, we deduce
\begin{align}\label{omegaj}
&\sup_{0\le t<\infty} \sum_{j=0}^1\norm{
w \p_t^j \omega^{l+1}(t)
}_{W^{1-j,p}(\Omega)}^p
  \nonumber\\&\quad\ls   \norm{ w \bar\p  \omega_0}_{L^p(\Omega)}^p+ \frac{c_2}{\alpha c_1}     \norm{ w \bar\p \omega_b^- }_{L^\infty(0,\infty;L^p(\gamma_-))}^p
 \nonumber\\&\qquad +   (c_1\alpha)^{-{p}}   e^{p\alpha} \norm{ \nabla  U
}_{W^{2,p}(\Omega)}^p    \sup_{0\le t<\infty}\sum_{j=0}^1\norm{
\p_t^j  v(t)
}_{W^{2-j,p}(\Omega)}^p
\nonumber\\&\qquad +    (c_1\alpha)^{-{p}} \norm{  \nabla  U
}_{W^{2,p}(\Omega)}^p \sup_{0\le t<\infty} \sum_{j=0}^1\norm{
w \p_t^j \omega^{l}(t)
}_{W^{1-j,p}(\Omega)}^p\nonumber
\\&\qquad +       (c_1\alpha)^{-{p}} e^{p\alpha}    \sup_{0\le t<\infty}\sum_{j=0}^1\norm{
\p_t^j  v(t)
}_{W^{2-j,p}(\Omega)}^p \sum_{j=0}^1\norm{
w \p_t^j \omega^{l}(t)
}_{W^{1-j,p}(\Omega)}^p.
\end{align}
Note that due to the factor $e^{p\alpha}$, in general, we could not expect to have smallness for the second term in the right hand side of \eqref{omegaj} by taking $\alpha$ large. So from now on, we fix $\alpha=1$. Then by \eqref{smallness}, we deduce from \eqref{omegaj} that
\begin{align}\label{omegaej}
&\sup_{0\le t<\infty}\sum_{j=0}^1\norm{
 \p_t^j \omega^{l}(t)
}_{W^{1-j,p}(\Omega)}^p
 \nonumber
 \\ &\quad\ls  \norm{
 \omega_0
}_{W^{1,p}(\Omega)}^p+      \sum_{j=0}^1\norm{
 \p_t^j \omega_b^-
}_{L^\infty(0,\infty;W^{1-j,p}(\gamma_-))}^p
  +    \delta \sup_{0\le t<\infty}\sum_{j=0}^1\norm{
\p_t^j  v(t)
}_{W^{2-j,p}(\Omega)}^p
\end{align}
by a simple induction on $l$. Moreover, $\{\omega^{l+1}-\omega^{l} \}$ satisfies
\begin{equation}
\begin{cases}
\p_t\{\omega^{l+1}-\omega^{l} \}+(u_s+v )\cdot\nabla \{\omega^{l+1}-\omega^{l} \} = \{\omega^{l}-\omega^{(l-1)} \} \cdot \nabla  (u_s +  v )\quad\text{in }\Omega,\\
\{\omega^{l+1}-\omega^{l} \}|_{t=0}=0,\quad \{\omega^{l+1}-\omega^{l} \}|_{\gamma_-}  =0.
\end{cases}
\end{equation}
Thus  $\omega^{l}$ is Cauchy in $L^{\infty}(0,\infty; L^p(\Omega))$:
\begin{align}
  \sup_{0\le t<\infty}\norm{    \{\omega^{l+1}-\omega^{l} \}(t)}_{L^p(\Omega)}^p & \ls\delta   \sup_{0\le t<\infty}\norm{    \{ \omega^{l+1}-\omega^{l}\}(t)  }_{L^p(\Omega)}
  \nonumber
  \\&<\frac{1}{2} \sup_{0\le t<\infty}\norm{     \{ \omega^{l+1}-\omega^{l}\}(t) }_{L^p(\Omega)}.
\end{align}
Then passing to the limit as $l\rightarrow\infty$ in \eqref{omegaj} leads to the solution $\omega$ to \eqref{omegan}, and the estimate \eqref{omegaes} follows from \eqref{omegaej}.

$\bullet$ {\it Proof of Part (2).} We follow the argument of \cite{GKM}.
Taking ${\rm div}$ to \eqref{omegan} yields, recalling $\dive v =\dive u_s=\dive \omega_s=0$,
\begin{align}  \label{ome11}
\p_t {\rm div }\omega  + (u_s+v)\cdot\nabla \dive\omega
=  \p_i(\omega_s +\omega)_j \p_j (u_s+v)_i-\p_i (u_s+v)_j  \p_j(\omega_s +\omega)_i =0.
\end{align}
Recalling the initial condition $\dive \omega_0=0$ in $\Omega$, it then suffices to guarantee $\dive \omega=0$ on $\gamma_-.$

Indeed, from the first component of \eqref{omegan}, using $\dive v =\dive u_s=0$, we deduce
\begin{align}\label{ome21}
\p_t  \omega_1+(u_s+v)_1 \dive \omega&+\p_2\left\{(\omega_s +\omega)_1(u_s+v)_2-(u_s+v)_1(\omega_s +\omega)_2\right\}\nonumber
\\& +\p_3\left\{(\omega_s +\omega)_1(u_s+v)_3-(u_s+v)_1(\omega_s +\omega)_3\right\}=0.
\end{align}
By taking the trace of \eqref{ome21} on (non-characteristic) $\gamma_-$, recalling $v_1=v_b^-$ and $\omega=\omega_b^-$ on $\gamma_-$ and $\omega_{b,1}^-=0$ on $\gamma_-$ from \eqref{bbcondition1}, we obtain
\begin{align}\label{ome212}
((u_s)_1+v_b^-) \dive \omega&+\p_2\left\{(\omega_s)_1(u_s+v)_2-((u_s)_1+v_b^-)(\omega_s+\omega_b^-)_2\right\}\nonumber
\\&+\p_3\left\{(\omega_s)_1(u_s+v)_3-((u_s)_1+v_b^-)(\omega_s+\omega_b^-)_3\right\}=0 \text{ on }\gamma_-.
\end{align}
Recalling $u_s=(U(x_2,x_3),0,0)$ and $\omega_s=(0,\p_{3}U,-\p_{2}U)$ so that
\[
(\omega_s)_1=0,\quad \p_2((u_s)_1 (\omega_s)_2)+\p_3((u_s)_1 (\omega_s)_3)=0 \text{ on }\gamma_-,
\]
 \eqref{ome212} then reduces to
 \beq\label{ome213}
(U+v_b^-) \dive \omega-\p_{2}(U\omega_{b,2}^-+v_b^-(\p_3U+\omega_{b,2}^-))-\p_{3}(U\omega_{b,3}^-+v_b^-(-\p_2U+\omega_{b,3}^-))=0\text{ on }\gamma_-.
\eeq
Now \eqref{ome213} together with the second condition in \eqref{bbcondition1} implies
\begin{equation}
(U+v_b^-) \dive \omega =0 \text{ on }\gamma_- \Longrightarrow \dive \omega =0 \text{ on }\gamma_-
\end{equation}
since $U+v_b^->0$. This together with $\dive \omega_0=0$ in $\Omega$ leads to \eqref{omegazero}, by Theorem \ref{scalar2}.

$\bullet$ {\it Proof of Part (3).}
We first consider the circular pipe case. Since $
\omega_s=0$ on $\gamma_0$, we have $\omega_s \cdot \nabla  v=0$ on $\gamma_0$ and also $
 v  \cdot\nabla   \omega_s=0$ on $\gamma_0$ due to that $v\cdot\nu=0$ on $\gamma_0$. Combing with the $\omega$-dependence of the last term in \eqref{omegan}, $\omega_0=0$ on $\gamma_0$ and $\omega_b^- =0$ on $\gamma_-\cap\gamma_0$, we conclude \eqref{oze22} upon using part (3) of Theorem \ref{scalar2} in the iteration above of constructing $\omega$.

Now we consider the square pipe case, and we shall prove \eqref{oze2} only for the piece $\{x_2=1\}\cap\gamma_0$, and the proof for the remaining pieces is the same. Noting $\nu=e_2$ on $\{x_2=1\}\cap\gamma_0$, so we are going to show $\omega_1=\omega_3=0$ on $\{x_2=1\}\cap\gamma_0$ under the assumptions $\omega_{s,1}=\omega_{s,3}=0$ on $\{x_2=1\}\cap\gamma_0$, $\omega_{0,1}=\omega_{0,3}=0$ on $\{x_2=1\}\cap\gamma_0$ and $\omega_{b,1}^-=\omega_{b,3}^-=0$ on $\gamma_-\cap\{x_2=1\}\cap\gamma_0$.
First, we have $
 v  \cdot\nabla   \omega_{s,1}= v  \cdot\nabla   \omega_{s,3}=0$ on $\{x_2=1\}\cap\gamma_0$ due to that $v\cdot\nu\equiv v_2=0$ on $\{x_2=1\}\cap\gamma_0$. Noting $\omega_s\cdot\nabla u_s=0$, we may rewrite the last two terms of \eqref{omegan} as
\beq
   \omega_s \cdot \nabla  v +  \omega  \cdot \nabla  (u_s +    v )
=  (\omega_s+ \omega ) \cdot \nabla  (u_s +  v ).\label{omega13}
\eeq
Since $(u_s +  v )\cdot \nu\equiv(u_s +  v )_2=0$ on $\{x_2=1\}\cap\gamma_0$,
 \begin{align}
 (\omega_s+ \omega ) \cdot \nabla  (u_s +  v )_1&=(\omega_s+ \omega)_1 \p_1 (u_s +  v )_1+(\omega_s+ \omega)_2 \p_2 (u_s +  v )_1+(\omega_s+ \omega)_3 \p_3 (u_s +  v )_1
\nonumber\\&=(\omega_s+ \omega)_1 \p_1 (u_s +  v )_1-(\omega_s+ \omega)_2 (\omega_s+ \omega)_3+(\omega_s+ \omega)_3 \p_3 (u_s +  v )_1
\nonumber\\&= \omega_1 \p_1 (u_s +  v )_1-(\omega_s+ \omega)_2 \omega_3+\omega_3 \p_3 (u_s +  v )_1 \nonumber
\end{align}
on $\{x_2=1\}\cap\gamma_0$ and similarly,
\beq
(\omega_s+ \omega ) \cdot \nabla  (u_s +  v )_3= \omega_1 \p_1 (u_s +  v )_3+(\omega_s+ \omega)_2 \omega_1+ \omega_3 \p_3 (u_s +  v )_3\nonumber
\eeq
on $\{x_2=1\}\cap\gamma_0$. These $\omega_1$-dependence and $\omega_3$-dependence of the first and third components of \eqref{omega13} together with the assumptions allow us to conclude \eqref{oze22} for the piece $\{x_2=1\}\cap\gamma_0$, similarly as the circular pipe case.
\end{proof}

\subsection{Div$-$Curl System}
Now we turn to the solvability and regularity of the div$-$curl system:
\begin{equation}
\begin{cases}
\curl v=\omega ,\quad \dive v =0\quad\!\text{ in }\Omega,
\\\dis v\cdot \nu=0 \text{ on }\gamma_0,\  v_1=v_{b,1}^\pm \text{ on }\gamma_\pm.
\end{cases}
\end{equation}
\begin{theorem}\label{vomth}
(1) Assume $\omega \in L^2(\Omega)$ with $\dive \omega=0$ in $\Omega$ and $v_{b,1}^\pm\in H^{1/2}(\gamma_\pm)$ satisfying \eqref{bbcondition}. Then there exists a unique solution $v\in H^1(\Omega)$ so that
\beq\label{dces1}
\norm{v}_{H^1(\Omega)}\ls \norm{\omega}_{L^2(\Omega)}+\norm{v_{b,1}^\pm}_{H^{1/2}(\gamma_\pm)}.
\eeq
(2) Assume further $\omega \in W^{1,p}(\Omega)$, $2\le p<\infty,$  $v_{b,1}^\pm\in W^{2-1/p,p}(\gamma_\pm)$ satisfying the condition in \eqref{oscondition31} for $D=(0,1)^2$ and \eqref{oscondition32} for $D=B(0,1)$, and $
\omega\times \nu=0$ on $\gamma_0$ for $D=(0,1)^2$  and $\omega =0$ on $\gamma_0$ for $D=B(0,1)$. Then
\beq\label{dces2}
\norm{v}_{W^{2,p}(\Omega)}\ls \norm{\omega}_{W^{1,p}(\Omega)}+\norm{v_{b,1}^\pm}_{W^{2-1/p,p}(\gamma_\pm)}.
\eeq
\end{theorem}
\begin{proof}
We first prove (1). We first homogenize the boundary condition. Thanks to \eqref{bbcondition}, we define $\psi$ with $\int_\Omega \psi=0$ as the unique $H^2$ solution to
\beq
\Delta\psi=0\text{ in }\Omega,\quad \p_1\psi= v_{b,1}^\pm\text{ on }\gamma_\pm,\ \nabla\psi\cdot \nu=0\text{ on }\gamma_0,
\eeq
see Theorem 3.2.1.2 in \cite{G1}; moreover,
\beq
\norm{ \psi}_{H^2(\Omega)}\ls  \norm{v_{b,1}^\pm}_{H^{1/2}(\gamma_\pm)}.
\eeq
We then solve $v=  v^\ast+\nabla\psi$ with $v^\ast$ as the solution to
\begin{equation}\label{vstareq}
\begin{cases}
\curl v^\ast=\omega ,\ \ \dive v^\ast=0\ \ \text{in }\Omega,\\
v^\ast \cdot\nu= 0 \  \text{on }\gamma_0,\ \,  v^\ast_1 =0   \ \text{on }\gamma_\pm.
\end{cases}
\end{equation}
Since $\dive \omega=0$ in $\Omega$, the existence of the unique $H^1$ solution $v^\ast$ to \eqref{vstareq} is guaranteed by Theorem 3.12 in \cite{ABDG}, which satisfies
\beq
\norm{v^\ast}_{H^1(\Omega)}\ls \norm{\omega}_{L^2(\Omega)}.
\eeq
We thus conclude \eqref{dces1}.

We will now prove (2) via a precise regular parity extension across $\gamma_0$.

$\bullet$ \underline{{\it Square Pipe:}}

We extend evenly $v_1$ (along with $v_{b,1}^\pm$), $v_3$, $\omega_2$ with respect to $\{x_2=\pm 1\}$, and since $v_2=\omega_1=\omega_3=0$ on $\{x_2=\pm 1\} $ we extend oddly $v_2$, $\omega_1$, $\omega_3$ with respect to $\{x_2=\pm 1\}$; for simplification of notation, the extensions are still denoted by $v,\omega$.  On the other hand, we further extend evenly $ v_1$ (along with $ v_{b,1}^\pm$), $ v_2$, $ \omega_3$ with respect to $\{x_3=\pm 1\}$, and since $v_3=\omega_1=\omega_2=0$ on $\{x_3=\pm 1\} $  we extend oddly $  v_3$, $  \omega_1$, $  \omega_2$ with respect to $\{x_3=\pm 1\}$; the extensions are now denoted by $\tilde v,\tilde\omega$ (and $\tilde v_{b,1}^\pm$). It is routine to check that
\begin{equation}\label{omegaeqeq3}
\begin{cases}
\curl  {\tilde v}= {\tilde\omega},\quad \dive  {\tilde v}=0\quad\text{in }  \tilde\Omega:=(-1,1)\times (-2,2)^2 ,\\ \dis
 {\tilde v}_1 = {\tilde v}_b^\pm   \quad \text{on } {\tilde \gamma}_\pm:=\{\pm1\}\times (-2,2)^2.
\end{cases}
\end{equation}
Moreover, by the extension, $ {\tilde \omega} \in W^{1,p}( {\tilde\Omega})$, $\tilde v\in H^1({\tilde\Omega})$,  $ {\tilde v}_b^\pm\in W^{2-1/p,p}(\tilde{ \gamma}_\pm)$, and
\beq\label{omes1}
  \norm{\tilde{  \omega}}_{W^{1,p}( {\tilde\Omega})}\ls \norm{\omega}_{W^{1,p}(\Omega)},\quad \norm{ { \tilde v}_b^\pm}_{W^{2-1/p,p}( {\tilde\gamma}_\pm)}\ls \norm{v_{b,1}^\pm}_{W^{2-1/p,p}(\gamma_\pm)},
\eeq
and by part (1),
\beq\label{omes2}
\norm{\tilde{  v}}_{H^1( {\tilde\Omega})}\ls \norm{v}_{H^1(\Omega)}\ls \norm{\omega}_{L^2(\Omega)}+\norm{v_{b,1}^\pm}_{H^{1/2}(\gamma_\pm)}.
\eeq

Now we choose a large smooth domain $E$ such that $(-1,1)\times (-7/4,7/4)^2\subset   E\subset  \tilde\Omega$, and take a smooth function $\chi=\chi(x_2,x_3)\in C_0^\infty((-7/4,7/4)^2)$ satisfying $\chi\equiv 1$ in $(-3/2,3/2)^2$. Then thanks to the cut-off function $\chi$, from \eqref{omegaeqeq3} we have
\begin{equation}\label{omegaeqeq4}
\begin{cases}
\curl(  \tilde v  \chi)= {\tilde\omega} \chi-{\tilde v}\times \nabla\chi  ,\quad \dive(  \tilde v   \chi)=\tilde v\cdot \nabla \chi   \quad \text{in }E,
\\ \tilde v  \chi\cdot \nu=\pm \tilde v_{b,1}^\pm   \chi \text{ on }\p E \cap \tilde\gamma_\pm,\quad \tilde v  \chi\cdot \nu= 0 \text{ on }\p E \cap \tilde\Omega,
\end{cases}
\end{equation}
where $\nu$ is the (smooth) unit outward normal to $E$ which takes $\pm e_1$ on $\p E \cap \tilde\gamma_\pm$.

We view $\tilde v \chi$ as the unique $H^1$ solution to the $div$--$curl$ system \eqref{omegaeqeq4} with the normal boundary condition.
Since now $E$ is a smooth domain, the classical $H^{2}(E)$ elliptic regularity estimate on \eqref{omegaeqeq4} (see Corollary 3.5 in \cite{AS}) together with \eqref{omes1} and \eqref{omes2}  implies
\begin{align}
  \norm{  v }_{H^2(\Omega)}\le \norm{\tilde v\chi}_{H^2(E)}& \ls \norm{ \tilde\omega }_{H^{1}(E)}
+\norm{ \tilde v }_{H^{1}(E)}+\norm{ \tilde v_{b,1}^\pm   \chi }_{H^{3/2}(\p E \cap \tilde\gamma_\pm )}
\nonumber
\\&  \ls \norm{ \omega  }_{H^{1}(\Omega)}+\norm{    v}_{H^{1}(\Omega)}+ \norm{v_{b,1}^\pm}_{H^{3/2}(\gamma_\pm)}\ls \norm{ \omega  }_{H^{1}(\Omega)} + \norm{v_{b,1}^\pm}_{H^{3/2}(\gamma_\pm)}.\nonumber
\end{align}
This yields \eqref{dces2} for $p=2$. Next, for $2<p\le 6$, by Sobolev's embedding, $v\in H^2(\Omega)\subset W^{1,p}(\Omega)$, and hence $\tilde v\in W^{1,p}(\tilde\Omega)$ by the extension. Then repeating the argument above by using the $W^{2,p}(E)$ regularity estimate on \eqref{omegaeqeq4} (see Corollary 3.5 in \cite{AS}), by \eqref{omes1} and \eqref{dces2} for $p=2$, we deduce
\begin{align}
\norm{v}_{W^{2,p}(\Omega)}\nonumber
&\ls  \norm{\omega}_{W^{1,p}(\Omega)}+\norm{    v}_{W^{1,p}(\Omega)}+\norm{v_{b,1}^\pm}_{W^{2-1/p,p}(\gamma_\pm)}
\\&\ls  \norm{\omega}_{W^{1,p}(\Omega)}+\norm{    v}_{H^{2}(\Omega)}+\norm{v_{b,1}^\pm}_{W^{2-1/p,p}(\gamma_\pm)}\ls  \norm{\omega}_{W^{1,p}(\Omega)}+\norm{v_{b,1}^\pm}_{W^{2-1/p,p}(\gamma_\pm)}.\nonumber
\end{align}
This yields \eqref{dces2} for $2<p\le 6$. Finally, for $6<p<\infty$, by Sobolev's embedding, $v\in W^{2,6}(\Omega)\subset W^{1,p}(\Omega)$ for $6<p<\infty$, and hence $\tilde v\in  W^{1,p}(\tilde\Omega)$ by the extension. Then repeating again the argument above, by \eqref{omes1} and \eqref{dces2} for $p=6$, we have
\begin{align}
\norm{v}_{W^{2,p}(\Omega)}\nonumber
&\ls  \norm{\omega}_{W^{1,p}(\Omega)}+\norm{    v}_{W^{1,p}(\Omega)}+\norm{v_{b,1}^\pm}_{W^{2-1/p,p}(\gamma_\pm)}
\\&\ls  \norm{\omega}_{W^{1,p}(\Omega)}+\norm{    v}_{W^{2,6}(\Omega)}+\norm{v_{b,1}^\pm}_{W^{2-1/p,p}(\gamma_\pm)}\ls  \norm{\omega}_{W^{1,p}(\Omega)}+\norm{v_{b,1}^\pm}_{W^{2-1/p,p}(\gamma_\pm)}.\nonumber
\end{align}
We thus conclude \eqref{dces2} for any $2\le p<\infty.$

$\bullet$ \underline{{\it Circular Pipe:}}

We want to extend $v,\omega$ (along with $v_{b,1}^\pm$) to be $\tilde v,\tilde \omega$ (and $\tilde v_{b,1}^\pm$), defined in a larger circular pipe $\tilde\Omega:=(-1,1)\times B(0,7/4)$ (and ${\tilde\gamma}_\pm:=\{\pm 1\}\times B(0,7/4)$). To this end, for $1/4\le r\le 1$,
we decompose in cylindrical coordinates
\[
 v (x_1,r,\theta)  =v_1 \boldsymbol{e}_1+ v_r \boldsymbol{e}_r + v_\theta \boldsymbol{e}_\theta ,\quad \omega(x_1,r,\theta)  =\omega_1 \boldsymbol{e}_1+ \omega_r \boldsymbol{e}_r + \omega_\theta \boldsymbol{e}_\theta .
\]
We recall
\begin{align}
&0={\rm div} v =\frac{\partial v_1}{\partial x_1}+ \frac{1}{r} \frac{\partial (r v_r)}{\partial r} + \frac{1}{r} \frac{\partial v_\theta}{\partial \theta }
= \frac{\partial v_1}{\partial x_1}+ \frac{\partial   v_r }{\partial r}+\frac{1}{r}  v_r  + \frac{1}{r} \frac{\partial v_\theta}{\partial \theta } \text{ in }\{ 1/4\le r< 1\},\label{oml0}
\\
&\omega={\rm curl} v =
\frac{1}{r} \left( \frac{\partial (r v_\theta)}{\partial r} - \frac{\partial v_r}{\partial \theta } \right) \boldsymbol{e}_1+
\left( \frac{1}{r} \frac{\partial v_1}{\partial \theta } - \frac{\partial v_\theta}{\partial x_1} \right) \boldsymbol{e}_r +
\left( \frac{\partial v_r}{\partial x_1} - \frac{\partial v_1}{\partial r} \right) \boldsymbol{e}_\theta \text{ in }\{ 1/4\le r< 1\}  \label{oml1}
\end{align}
and
\begin{equation}
0=v\cdot \nu= v\cdot \boldsymbol{e}_r=v_r,\quad \omega=0\text{ on }\{r=1\}.\label{oml2}
\end{equation}
We then extend evenly $v_1$ (along with $v_{b,1}^\pm$), $rv_\theta, r\omega_r$ with respect to $\{r=1\}$, and extend oddly $v_r,\omega_\theta,r\omega_1$; the extensions are denoted $\tilde v_1,\tilde v_r,\tilde v_\theta,\tilde \omega_1,\tilde \omega_r,\tilde\omega_\theta$, and we then define
\[
 \tilde\omega   =\tilde \omega_1 \boldsymbol{e}_1+ \tilde \omega_r \boldsymbol{e}_r + \tilde \omega_\theta \boldsymbol{e}_\theta ,\quad \tilde v    =\tilde v_1 \boldsymbol{e}_1+ \tilde v_r \boldsymbol{e}_r + \tilde v_\theta \boldsymbol{e}_\theta.
 \]

We claim
\begin{align}\label{omom}
{\rm curl} \tilde v= \tilde\omega,\quad{\rm div}  \tilde v   =\sum_{(i,j)=(2,3)}\zeta_{ij}  \p_{i}\tilde v_j +  \zeta \cdot\tilde v  \text{ in }\tilde \Omega,
\end{align}
where
\beq\label{zijf}
\zeta_{ij}=\zeta_{ij}(x_2,x_3)\in W^{1,\infty}(\{r\le 7/4\})\text{ satisfying }\zeta_{ij}=0 \text{ in } \{0\le r\le 1\}
\eeq
and
\beq\label{zf}
\zeta=\zeta(x_2,x_3)\in W^{1,\infty}(\{r\neq 1\})\text{ satisfying } \zeta=0\text{ in }\{0\le r< 1\}\text{ and }
\zeta\cdot \tilde v =0 \text{ on }  \{  r= 1\}.
\eeq
{\it Proof of the claim.} For $1<r<7/4$, we compute, by \eqref{oml1},
\begin{align}
({\rm curl}  \tilde{v})_1(x_1,r,\theta)&=\left[ \frac{1}{r}\left(\frac{\partial (r \tilde v_\theta)}{\p r}  - \frac{\partial \tilde v_r}{\partial \theta }\right) \right] (x_1,r,\theta)
\nonumber \\& =
\frac{1}{r} \left[ -\frac{\partial (r   v_\theta)}{\partial r} (x_1,2-r,\theta) + \frac{\partial   v_r}{\partial \theta } (x_1,2-r,\theta)\right]
\nonumber\\& =-
\frac{2-r}{r} \left[\frac{1}{r} \left(\frac{\partial (r   v_\theta)}{\partial r}  - \frac{\partial   v_r}{\partial \theta }\right) \right] (x_1,2-r,\theta) =\tilde \omega_1(x_1,r,\theta),\label{curl1}\\
({\rm curl} \tilde{v})_r(x_1,r,\theta)&=\left[ \frac{1}{r}\frac{\partial \tilde{v}_1}{\partial \theta }  - \frac{\partial \tilde{v}_\theta}{\partial x_1}\right](x_1,r,\theta)\nonumber
\\& =
\frac{1}{r} \frac{\partial  v_1}{\partial \theta }(x_1,2-r,\theta) - \frac{2-r}{r}\frac{\partial   {v}_\theta}{\partial x_1}(x_1,2-r,\theta)\nonumber
\\& = \frac{2-r}{r}\left[
\frac{1}{r} \frac{\partial  v_1}{\partial \theta } -  \frac{\partial   {v}_\theta}{\partial x_1} \right](x_1,2-r,\theta)
=\tilde{\omega}_r(x_1,r,\theta),\label{curl2}
\\
({\rm curl}  \tilde{v})_\theta(x_1,r,\theta)&=\left[\frac{\partial \tilde v_r}{\partial x_1}  - \frac{\partial \tilde v_1}{\partial r}\right](x_1,r,\theta)
\nonumber\\& =
-\frac{\partial   v_r}{\partial x_1}(x_1,2-r,\theta)+ \frac{\partial  v_1}{\partial r}(x_1,2-r,\theta)
\nonumber\\& =-\left[\frac{\partial v_r}{\partial x_1}- \frac{\partial v_1}{\partial r}\right](x_1,2-r,\theta)
=  \tilde{\omega}_\theta(x_1,r,\theta).\label{curl3}
\end{align}
On the other hand, by \eqref{oml0},
\begin{align}\label{vthe0}
\dive \tilde v (x_1,r,\theta) &= \left[ \frac{\partial \tilde v_1}{\partial {x_1}}+\frac{\partial  \tilde v_r}{\partial r} +\frac{1}{r}\tilde v_r  + \frac{1}{r} \frac{\partial \tilde v_\theta}{\partial \theta}  \right](x_1,r,\theta)\nonumber
 \\&=\left[ \frac{\partial   v_1}{\partial {x_1}}+\frac{\partial    v_r}{\partial r}\right] (x_1,2-r,\theta)-\frac{1}{r}  v_r(x_1,2-r,\theta) + \frac{1}{r}  \frac{2-r}{r}\frac{\partial   v_\theta}{\partial \theta} (x_1,2-r,\theta) .\nonumber
 \\&=-\left[ \frac{1}{r}  v_r   + \frac{1}{r} \frac{\partial v_\theta}{\partial \theta} \right](x_1,2-r,\theta) -\frac{1}{r}  v_r (x_1,2-r,\theta) + \frac{1}{r}  \frac{2-r}{r}\frac{\partial   v_\theta}{\partial \theta}(x_1,2-r,\theta)\nonumber
 \\&=\left[\left(\frac{1}{r}  -\frac{r}{(2-r)^2}\right)  \frac{\partial \tilde v_\theta}{\partial \theta}  + \left(\frac{1}{2-r}+\frac{1}{r}\right)\tilde v_r\right] (x_1,r,\theta).
\end{align}
Recalling that
\beq\label{vthe1}
\frac{\partial \tilde v_\theta}{\partial \theta}
= \frac{1}{r}\left[
x_2^{2}\frac{\partial \tilde v_2}{\partial x_2}
- x_2x_3\left(\frac{\partial \tilde v_2}{\partial x_3}+\frac{\partial \tilde v_{3}}{\partial x_2}\right)
+ x_2^{2}\frac{\partial \tilde v_2}{\partial x_3}
- \bigl(x_2 \tilde v_2 + x_3 \tilde v_3\bigr)\right],
\eeq
and by \eqref{oml2},
\beq\label{vthe2}
\tilde v_r = \frac{1}{r}(x_2\,\tilde v_2 + x_3\,\tilde v_3),\quad \tilde v_r=-v_r=0\text{ on }\{r=1\},
\eeq
\eqref{curl1}--\eqref{curl3} and inserting \eqref{vthe1} and \eqref{vthe2} into \eqref{vthe0} lead to the equations \eqref{omom} with \eqref{zijf} and \eqref{zf}. It should be emphasized that $\zeta_{ij}=0$ on $\{r= 1\}$ is due to the factor $\frac{1}{r}  -\frac{r}{(2-r)^2}$ in \eqref{vthe0} which vanishes at $r=1$, and $\zeta\cdot \tilde v =0$ on $\{  r= 1\}$ is due to $v\cdot \nu=0$ on $\{r=1\}$. The proof of the claim is thus completed.

Moreover, by the extension, ${\tilde \omega} \in W^{1,p}({\tilde\Omega})$, $\tilde v\in H^1({\tilde\Omega})$,  ${\tilde v}_b^\pm\in W^{2-1/p,p}({\tilde\gamma}_\pm)$, and
\beq\label{gg12}
  \norm{{\tilde \omega}}_{W^{1,p}({\tilde\Omega})}\ls \norm{\omega}_{W^{1,p}(\Omega)},\quad \norm{{\tilde v}_b^\pm}_{W^{2-1/p,p}({\tilde\gamma}_\pm)}\ls \norm{v_{b,1}^\pm}_{W^{2-1/p,p}(\gamma_\pm)},
\eeq
and by part (1),
\beq\label{gg13}
\norm{{\tilde v}}_{H^1({\tilde\Omega})}\ls \norm{v}_{H^1(\Omega)}\ls \norm{\omega}_{L^2(\Omega)}+\norm{v_{b,1}^\pm}_{H^{1/2}(\gamma_\pm)}.
\eeq

Now we choose a large smooth domain $E$ such that $(-1,1)\times B(0,7/4)\subset   E\subset  \tilde\Omega$, and take a smooth function $\chi=\chi(x_2,x_3)\in C_0^\infty(B(0,1+2\delta))$ satisfying $\chi\equiv 1$ in $B(0,1+\delta)$ for sufficiently small $0<\delta<1/4$. Again, thanks to the cut-off function $\chi$, from \eqref{omom} we have
\begin{equation}\label{omegaeqeq4s}
\begin{cases}
\curl(\tilde{ v} \chi)= \tilde{\omega} \chi-\tilde{ v}\times\nabla\chi  ,\quad \dive(\tilde{ v } \chi)=  \zeta_{ij}    \p_{i}(\tilde v\chi)_j-\zeta_{ij}\tilde v_j\p_{i}\chi+\zeta \cdot\tilde v \chi +\tilde{ v}\cdot \nabla\chi  \text{ in }E,
\\ \tilde v  \chi\cdot \nu=\pm \tilde v_{b,1}^\pm   \chi \text{ on }\p E \cap \tilde\gamma_\pm,\quad \tilde v  \chi\cdot \nu= 0 \text{ on }\p E \cap \tilde\Omega,
\end{cases}
\end{equation}
where $\nu$ is the (smooth) unit outward normal to $E$ which takes $\pm e_1$ on $\p E \cap \tilde\gamma_\pm$.

{\it A Prior $W^{2,p}$ Estimate for \eqref{omegaeqeq4s}.}
It is essentially crucial that $|\zeta_{ij}|\ls\delta $ in the support of $\chi$ by \eqref{zijf} due to $\zeta=0$ on $\{r=1\}$, and we then view $\tilde v \chi$ as the unique $H^1$ solution to the $div$--$curl$ system \eqref{omegaeqeq4s} with the normal boundary condition for such sufficiently small $\delta$.
We now prove the higher regularity estimates. It is important that $\zeta \cdot\tilde v \in W^{1,p}(\tilde \Omega)$ (even $\zeta\notin W^{1,p}$) if $v\in W^{1,p}(\Omega)$ by \eqref{zf} due to $\zeta \cdot\tilde v=0$ on $\{r=1\}$. Since now $E$ is a smooth domain, the classical $H^{2}(E)$ regularity estimates  on \eqref{omegaeqeq4s} (see Corollary 3.5 in \cite{AS}) together with \eqref{gg12} and \eqref{gg13} implies, recalling $|\zeta_{ij}|\ls\delta $,
\begin{align}
\norm{\tilde{ v} \chi}_{H^2(E)}
 &\ls  \norm{ \zeta_{ij}    \p_{i} ( \tilde{ v} \chi)_j}_{H^{1}(E)}+ \norm{ \tilde{\omega} }_{H^{1}(E)}
+\norm{  \tilde{ v}}_{H^{1}(E)}+\norm{ \tilde{ v}_b^\pm}_{H^{3/2}(\tilde{\gamma}_\pm)}
\nonumber
\\&  \ls  \delta\norm{\tilde{ v} \chi}_{H^2(E)}+ \norm{ \omega }_{H^{1}(\Omega)}
 +\norm{  v_{b,1}^\pm}_{H^{3/2}( {\gamma}_\pm)} ,\nonumber
\end{align}
which implies, for fixed $\delta$ sufficiently small,
\[
\norm{  v }_{H^2(\Omega)}\le\norm{\tilde{ v} \chi}_{H^2(E)} \ls   \norm{  {\omega} }_{H^{1}(\Omega)}
 +\norm{v_{b,1}^\pm}_{H^{3/2}( \gamma_\pm)}.
\]
This yields the {\it a priori} estimate \eqref{dces2} for $p=2$.
On the other hand, repeating  the argument above in instead $W^{2,6}$ and $W^{2,p} $ with $6<p<\infty$ by letting $\delta$ be sufficiently small, we can bootstrap similarly as the square pipe case to conclude the {\it a priori} estimate \eqref{dces2} for any $2\le p<\infty.$

Thanks to $\delta\ll1$, we can justify such a priori estimates by a standard contraction mapping argument in both $H^2$ and $W^{2,p}$ as well as the uniqueness of solution to the problem \eqref{omegaeqeq4s}, upon replacing $\zeta_{ij}$ by $\tilde \chi \zeta_{ij}$ with $\tilde \chi=\tilde \chi(x_2,x_3)\in C_0^\infty(B(0,1+3\delta))$ satisfying $\tilde \chi\equiv 1$ in $B(0,1+2\delta)$ (and in the support of $\chi$) to keep the smallness $|\tilde \chi \zeta_{ij}|\ls \delta$ in the argument.
\end{proof}

\subsection{3D Euler System}

We are now ready to prove our main Theorem \ref{th2}.

\begin{proof}[\bf Proof of Theorem \ref{th2}]
~

{\it Step 1. Vorticity Formula.}
We will first construct the solution $(\omega, v)$ to the vorticity formula \eqref{omegaeq} and \eqref{omegaeqv} by the iteration. We start with $\omega^{(1)}=0$.  Suppose $\omega^{(n)}$ is given with
$
 \dive \omega^{(n)}=0$ in $\Omega$,
$
\omega^{(n)}\times \nu=0$ on $\gamma_0
$ for $D=(0,1)^2$ and  $
\omega^{(n)}=0$ on $\gamma_0$ for $D=B(0,1)$, we construct $v^{(n)}$ as the solution to the {\it div}--{\it curl} problem
\begin{equation} \label{last1}
\begin{cases}
{\rm curl\,}v^{(n)}=\omega^{(n)},\  \dive v^{(n)}=0\ \,\text{in }\Omega,
\\ v^{(n)}\cdot\nu= 0 \  \text{on }\gamma_0 ,\ \, v_1^{(n)}=v_{b,1}^\pm   \ \text{on }\gamma_\pm.
\end{cases}
\end{equation}
With such $v^{(n)}$, we then construct $\omega^{(n+1)}$ as the solution to the vorticity equations
\begin{equation} \label{last2}
\begin{cases}
\p_t\omega^{(n+1)}+(u_s+v^{(n)})\cdot\nabla \omega^{(n+1)} =- v^{(n)} \cdot\nabla   \omega_s+  \omega_s \cdot \nabla  v^{(n)}+  \omega^{(n+1)} \cdot \nabla ( u_s +    v^{(n)})\quad\text{in }\Omega,\\
\omega^{(n+1)}|_{t=0}=\omega_0,\quad\omega^{(n+1)}|_{\gamma_-}=\omega_b^-.
\end{cases}
\end{equation}

$\bullet$ \textit{Induction Hypothesis:} We assume that
\beq\label{omegaesnn}
\sup_{0\le t<\infty}\sum_{j=0}^1 \norm{
 \p_t^j \omega^{(n)}(t)
}_{W^{1-j,p}(\Omega)}^p\le \delta\ll1.
\eeq

Beginning with such $\omega^{(n)}$ and recalling \eqref{bbcondition}, the solvability of \eqref{last1} for $v^{(n)}$ is guaranteed by Theorem \ref{vomth}:
 \begin{align}\label{vn}
&\sup_{0\le t<\infty}\sum_{j=0}^1\norm{
\p_t^j  v^{(n)}(t)
}_{W^{2-j,p}(\Omega)}^p\nonumber
\\&\quad
\ls  \sup_{0\le t<\infty}\sum_{j=0}^1\norm{
 \p_t^j \omega^{(n)}(t)
}_{W^{1-j,p}(\Omega)}^p+ \sup_{0\le t<\infty}\sum_{j=0}^1  \norm{\partial_t^j v_{b,1}^\pm(t)}_{W^{2-j-1/p,p}(\gamma_\pm)}^p.
\end{align}
Here we have also invoked Theorem \ref{vomth} for $\p_t v^{(n)}$.

Next, with such $v^{(n)}$ in hand, assuming $\eps_0$ in \eqref{small2} sufficiently small, $\delta\ll1$  and recalling from \eqref{thcom} the compatibility condition $\omega_0|_{\gamma_-}=\omega_b^-(0)$, the solvability of \eqref{last2} for $\omega^{(n+1)}$ is guaranteed by part (1) of Theorem \ref{vorth}:
\begin{align}\label{omegaesn}
&\sup_{0\le t<\infty}\sum_{j=0}^1\norm{
 \p_t^j \omega^{(n+1)}(t)
}_{W^{1-j,p}(\Omega)}^p
\nonumber\\&\quad\ls  \norm{
 \omega_0
}_{W^{1,p}(\Omega)}^p+     \sum_{j=0}^1\norm{
 \p_t^j \omega_b^-
}_{L^\infty(0,\infty;W^{1-j,p}(\gamma_-))}^p +    \delta\sup_{0\le t<\infty}\sum_{j=0}^1\norm{
\p_t^j  v^{(n)}(t)
}_{W^{2-j,p}(\Omega)}^p      .
\end{align}
Since $\delta\ll1,$ \eqref{vn} and \eqref{omegaesn} imply, recalling $\omega_0=\curl v_0$,
\begin{align} \label{omegaesnn+1}
 &\sup_{0\le t<\infty}\sum_{j=0}^1\left\{ \norm{
 \p_t^j \omega^{(n+1)}
}_{W^{1-j,p}(\Omega)}^p+\norm{
 \p_t^j v^{(n+1)}
}_{W^{2-j,p}(\Omega)}^p\right\}
\\  &\quad\ls   \norm{
v_0
}_{W^{2,p}(\Omega)}^p+\sup_{0\le t<\infty} \sum_{j=0}^1 \norm{\partial_t^j v_{b,1}^\pm(t)}_{W^{2-j-1/p,p}(\gamma_\pm)}+\sum_{j=0}^1\norm{\partial_t^j\omega_b^-}_{L^\infty(0,\infty;W^{1-j,p}(\gamma_-))}\le \delta\nonumber
\end{align}
by induction over $n$ and choosing $\eps_0$ in \eqref{small2} sufficiently small. Moreover, recalling \eqref{bbcondition1} and \eqref{oscondition31}--\eqref{oscondition320}, by parts (2) and (3) of Theorem \ref{vorth}, we have that $
 \dive \omega^{(n+1)}=0$ in $\Omega$,
$
\omega^{(n+1)}\times \nu=0$ on $\gamma_0
$ for $D=(0,1)^2$ and  $
\omega^{(n+1)}=0$ on $\gamma_0$ for $D=B(0,1)$.  We therefore finish the construction of the sequence of approximate solutions $(\omega^{(n)},v^{(n)})$, and establish  the uniform bound \eqref{omegaesnn+1} by induction on $n$.

On the other hand, taking difference yields
\begin{equation}
\begin{cases}
\p_t\{\omega^{(n+1)}-\omega^{(n)}\}+(u_s+v^{(n)})\cdot\nabla\{\omega^{(n+1)}-\omega^{(n)}\}
\\\quad=
-(v^{(n)}-v^{(n-1)})\cdot\nabla (\omega_s+\omega^{(n)})+ (\omega_s+\omega^{(n)})\cdot \nabla (v^{(n)}-v^{(n-1)})
\\\qquad+  \{\omega^{(n+1)}-\omega^{(n)}\} \cdot \nabla  (u_s +  v^{(n)}) ,\\
\{\omega^{(n+1)}-\omega^{(n)}\}|_{t=0}=0,\quad \{\omega^{(n+1)}-\omega^{(n)}\}|_{\gamma_-} =0.
\end{cases}
\end{equation}
We apply the estimate \eqref{esf} with $p=2$ to obtain
\begin{align}
&\sup_{0\le t<\infty}\norm{\{\omega^{(n+1)}-\omega^{(n)}\}(t)}_{L^2(\Omega)}^2
\nonumber\\&\quad\ls \delta  \sup_{0\le t<\infty}\left\{ \norm{\{\omega^{(n+1)}-\omega^{(n)}\}(t)
}_{L^2(\Omega)}^2 +\norm{\{v^{(n)}-v^{(n-1)}\}(t)}_{H^1(\Omega)}^2 \right\} .
\end{align}
Also,
\begin{equation}
\begin{cases}
\curl \{v^{(n)}-v^{(n-1)}\}=\omega^{(n)}-\omega^{(n-1)},\quad \dive \{v^{(n)}-v^{(n-1)}\}=0\quad\text{in }\Omega,\\ \dis
\{v^{(n+1)}-v^{(n)}\} \cdot\nu= 0 \quad \text{on }\gamma_0 ,\quad \{v^{(n)}-v^{(n-1)}\}_1=0   \quad \text{on }\gamma_\pm.
\end{cases}
\end{equation}
By \eqref{dces1} in Theorem \ref{vomth},
\begin{align}
\norm{v^{(n)}-v^{(n-1)} }_{H^1(\Omega)} \ls \norm{\omega^{(n)}-\omega^{(n-1)} }_{L^2(\Omega)}   .
\end{align}
Therefore, for $\delta\ll 1$,
\begin{align}\label{contra}
&\sup_{0\le t<\infty}\left\{\norm{\{v^{(n+1)}-v^{(n)}\}(t)}_{H^{1}(\Omega)}  + \norm{\{\omega^{(n+1)}-\omega^{(n)}\}(t)}_{L^2(\Omega)} \right\}
\nonumber\\&\quad\le \frac{1}{2}\sup_{0\le t<\infty} \norm{\{\omega^{(n)}-\omega^{(n-1)}\}(t)}_{L^{2}(\Omega)} .
\end{align}
We then conclude that $\{(\omega^{(n)}, v^{(n)})\} $ is Cauchy. Then passing to the limit as  $n\rightarrow\infty$, we get the solution $(\omega, v)$ to the vorticity formula \eqref{omegaeq} and \eqref{omegaeqv}, with the estimate \eqref{goodes} from \eqref{omegaesnn+1}.

{\it Step 2. Recovering Euler System.} It remains to recover the momentum equations in \eqref{eulerv} and verify the initial condition $v|_{t=0}=v_0$.  We recall $\omega=\curl v$ from \eqref{omegaeqv} and then plug it into the vorticity equations \eqref{omegaeq} to find
\[
\p_t{\rm curl }v+(u_s+v)\cdot\nabla ({\rm curl\,} u_s+{\rm curl }v)=({\rm curl\,} u_s+{\rm curl }v)\cdot \nabla (u_s+v),
\]
which can be rewritten as
\[
{\rm curl }\left\{\p_t v +(u_s+v)\cdot\nabla v + v \cdot\nabla  u_s \right\}=0.
\]
It is then classical that there exists $p$ such that
\begin{equation}
\p_t v +(u_s+v)\cdot\nabla v + v \cdot\nabla  u_s=-\nabla p.
\end{equation}
On the other hand, valuing \eqref{omegaeqv} at $t=0$  yields
\begin{equation}\label{last10}
\begin{cases}
\curl v(0)=\omega(0),\quad \dive v(0) =0\quad\text{in }\Omega,\\ \dis
v(0) \cdot\nu= 0 \quad \text{on }\gamma_0 ,\quad v_1(0) =v_{b,1}^\pm(0)   \quad \text{on }\gamma_\pm.
\end{cases}
\end{equation}
Recalling the initial data $\omega(0)=\omega_0\equiv\curl v_0$,
and from \eqref{thcom} the compatibility conditions $\dive v_0=0$ in $\Omega$, $v_0\cdot \nu=0$ on $\gamma_0$ and $v_{0,1}=v_{b,1}^\pm(0)$ on $\gamma_\pm$, we thus have that $v(0)=v_0$ from the uniqueness of the $div$--$curl$ problem \eqref{last10}. We thus conclude \eqref{eulerv}.

{\it Step 3. Uniqueness.} The uniqueness of solutions to \eqref{eulerv} and \eqref{eulervBC} follows in the same way as leading to the contraction \eqref{contra} in the iteration above.
\end{proof}

\appendix
\section{3D Euler with Velocity Inflow}

There are three classical types of BCs for the incompressible Euler system \eqref{eulerv} through the pipe $\Omega$, namely, \eqref{eulervBC} and the following two \cite{AKM,GKM,KR1,KR,K,Y,Z1,Z2}:
\begin{align}
&v\cdot \nu=0 \text{ on }\gamma_0,\  \  v =v_b^-\text{ on }\gamma_-,\ \   p =p_b^+ \text{ on }\gamma_+;\label{BC2}\\
&v\cdot \nu=0 \text{ on }\gamma_0,\  \  v =v_b^-\text{ on }\gamma_-,\ \   v_1 =v_{b,1}^+ \text{ on }\gamma_+.\label{BC3}
\end{align}
In this appendix we will show the stability of shear flows for the 3D incompressible Euler system \eqref{eulerv} with the BC \eqref{BC2}, or \eqref{BC3} of a narrow pipe. The global regular square pipe flows are 3D, while circular pipe flows will be restricted to be axisymmetric with general nonzero swirls from initial-boundary inputs (zero on the lateral sides), in contrast to those in \cite{CH0,CH1,CH}. Such axisymmetry is necessary for compatibility at $\gamma_+\cap\gamma_0$ for the problem \eqref{eulerv} and \eqref{BC2} (see Remark 2 in \cite{KR}), while it is needed for \eqref{BC3} to ensure the crucial property $\omega=0$ on $\gamma_0$ as for Theorem \ref{th2}. Our new interior dissipation mechanism leads to a priori estimate for global stability, while we only outline the iterative constructions for completeness in light of classical works \cite{AKM,GKM,KR1,KR,Z2}.

A subtle technical issue for the study of these two types of BCs is to establish in addition the $W^{3,p}$ regularity theory for the Poisson equation:
\begin{align}\label{ppaeq}
\dis- \Delta  p  = F \quad\text{in }\Omega
\end{align}
with accordingly the following two classical types of BCs:
\begin{align}
& \nabla p\cdot \nu=0 \text{ on }\gamma_0,\ \ \p_1    p=  g_b^-\text{ on }{  \gamma}_-, \ \  p=  p_b^+\text{ on }{  \gamma}_+;\label{BC4}\\
&\nabla p\cdot \nu=0 \text{ on }\gamma_0,\ \ \p_1    p=  g_b^-\text{ on }{  \gamma}_-, \ \  \p_1 p=  g_b^+\text{ on }{  \gamma}_+ .\label{BC5}
\end{align}

\begin{theorem}\label{vomthp1}
(1) Assume $F\in L^p(\Omega) ,\ g_b^-\in W^{1-1/p,p}(  \gamma_-)$ and $ p_b^+\in W^{2-1/p,p}(  \gamma_+)$ with
$\nabla p_b^+\cdot\nu=0$ on $\gamma_+\cap\gamma_0 .
$
Then there exists a unique solution $p\in W^{2,p}(\Omega)$ to \eqref{ppaeq} and \eqref{BC4}:
\begin{align}\label{pes1}
\norm{p}_{W^{2,p}(\Omega)}
\ls \norm{ F}_{L^{p}(  \Omega)}+\norm{  g_b^-}_{W^{1-1/p,p}(  \gamma_-)} +\norm{  p_b^+}_{W^{2-1/p,p}(  \gamma_+)}.
\end{align}
(2) Assume $F \in W^{1,p}(\Omega)$, $g_b^-\in W^{2-1/p,p}(  \gamma_-)$ with $\nabla g_b^-\cdot\nu=0$ on $\gamma_- \cap \gamma_0$ and $ p_b^+\in W^{3-1/p,p}(  \gamma_+)$ with $\nabla p_b^+\cdot\nu=0$ on $\gamma_+\cap\gamma_0 .$
Then
\begin{align}\label{pes2}
\norm{p}_{W^{3,p}(\Omega)}
\ls \norm{ F}_{W^{1,p}(  \Omega)}+\norm{  g_b^-}_{W^{2-1/p,p}(  \gamma_-)} +\norm{  p_b^+}_{W^{3-1/p,p}(  \gamma_+)}.
\end{align}
\end{theorem}
\begin{proof}
The existence of unique $H^1$  weak solution $p$ is standard, and in spirit of the proof of Theorem \ref{vomth} we can then prove the higher regularity of $p$ via the even extension across $\gamma_0$ combined with a cutoff arument; $\nabla p_b^+\cdot\nu=0\text{ on }  \gamma_+\cap\gamma_0
$ (resp. together with $\nabla g_b^-\cdot\nu=0$ on $\gamma_- \cap \gamma_0$) is necessary for the $W^{2,p}$ estimate \eqref{pes1} (resp. $W^{3,p}$ estimate \eqref{pes2}). Note that the BC $\nabla p\cdot \nu=0$ on $\gamma_0$ is also guaranteed by the even extension.
\end{proof}
\begin{theorem}\label{vomthp2}
(1) Assume $F \in L^p(\Omega)$ and $g_b^\pm\in W^{1-1/p,p}(  \gamma_\pm)$ with
$
\int_\Omega F=\int_{\gamma_-}g_b^--\int_{\gamma_+}g_b^+.
$
Then there exists a unique solution $p\in W^{2,p}(\Omega)$ with $\int_\Omega p =0$  to \eqref{ppaeq} and \eqref{BC5}:
\begin{align}\label{pes11}
\norm{p}_{W^{2,p}(\Omega)}
\ls \norm{F}_{L^{p}(  \Omega)}+\norm{  g_b^\pm}_{W^{1-1/p,p}(  \gamma_\pm)}  .
\end{align}
(2) Assume $F\in W^{1,p}(\Omega)$, $g_b^-\in W^{2-1/p,p}(  \gamma_-)$ with
$
\nabla g_b^-\cdot\nu=0\text{ on }\gamma_- \cap \gamma_0
$ and $g_b^-\in L^{2}(\gamma_+)$.
Then, denoting $\Omega_-:=\Omega\cap\{x_1\le 0\}$,
\begin{align}\label{pes22}
\norm{p}_{W^{3,p}(\Omega_-)}
\ls \norm{F}_{W^{1,p}(  \Omega)}+\norm{  g_b^-}_{W^{2-1/p,p}(  \gamma_-)} +\norm{  g_b^+}_{L^{2}(  \gamma_+)} .
\end{align}
\end{theorem}
\begin{proof}
The proof follows in the same way as Theorem \ref{vomthp1}, and we utilize addtionally a localized argument away from $\gamma_+$ for establishing the estimate \eqref{pes22}.
\end{proof}

In the application of the  two theorems to the incompressible Euler system \eqref{eulerv},
\begin{align}\label{ppf}
&\dis F =F(v):=  2\sum_{i=2,3}\p_iU\p_1v_i +\sum_{i,j=2,3}\p_iv_j \p_jv_i  \text{ in }\Omega,
\\\label{pbbc1}
&g_b^-:=-\partial_t v^-_{b,1} +(U+v^-_{b,1})\sum_{i=2,3} \p_i v^-_{b,i}
- \sum_{i=2,3}v^-_{b,i}(\p_i v^-_{b,1}+ \p_i U)\text{ on }\gamma_-,
\\ &\label{pbbc2}
g_b^+=g_b^+(v):=-\partial_t v^+_{b,1} +(U+v^+_{b,1})\sum_{i=2,3} \p_i v_{i}
- \sum_{i=2,3}v_{i}(\p_i v^+_{b,1}+ \p_i U)\text{ on }\gamma_+,
\end{align}
and the circular pipe flows will be restricted to be axisymmetric so that the BC $\nabla p\cdot \nu=0$ on $\gamma_0$ can be enforced by requiring no swirl on $\gamma_0$ due to the curved nature.

\subsection{Velocity-pressure BC Problem}The main result of this section is
\begin{theorem}\label{2th2}
Let $3<p<\infty$. Assume $v_{b}^-\cdot\nu=0$ and $\nabla  g_b^-\cdot\nu=0$ on $\gamma_- \cap \gamma_0$, $\nabla p_b^+\cdot\nu=0$ and $\nabla \p_t p_b^+\cdot\nu=0$ on $\gamma_+ \cap \gamma_0$, and the compatibility conditions:
\begin{align}\label{thcom11}
&\dive v_0=0\text{ in }\Omega,\ v_0\cdot \nu=0\text{ on }\gamma_0,\  v_{0}=v_b^-(0) \text{ on }\gamma_-,
\\\label{thcom22}
&-(u_s+v_0)\cdot\nabla v_{0,i} -  \p_i  p_0=\p_tv_{b,i}^-(0)\text{ on }\gamma_-, \ i=2,3 ,
\end{align}
where  $p_0$ is the unique solution to
\begin{align}\label{omegaeq0}
- \Delta  p_0  =  F(v_0) \ \text{in }\Omega,
\ \nabla p_0\cdot \nu=0 \text{ on }\gamma_0,\   \p_1    p_0=  g_b^-(0)\text{ on }{  \gamma}_-, \    p_0=  p_b^+(0)\text{ on }{  \gamma}_+.
\end{align}
For the axisymmetric circular pipe flow,  assume further $v_{0,\theta}=0\text{ on }\gamma_0$ and $
v_{b,\theta}^- =0\text{ on }\gamma_-\cap \gamma_0$.
There exists a constant  $\delta>0$ such that if
$
\norm{\nabla U}_{W^{2,p}(\Omega)}< \delta  \min_{\bar\Omega} U ,
$
then there is a unique solution to  \eqref{eulerv} and \eqref{BC2} on $[0,\infty)$ satisfying
\begin{align}\label{2goodes}
&\sup_{0\le t<\infty} \sum_{j=0}^{1}\norm{\p_t^j v(t)}_{W^{2-j,p}(\Omega)}
\\&\quad \lesssim \norm{v_0}_{W^{2,p}(\Omega)}+\sup_{0\le t<\infty} \left\{\sum_{j=0}^2 \norm{\partial_t^j v_b^-(t)}_{W^{3-j-1/p,p}(\gamma_-)}+\sum_{j=0}^1\norm{\partial_t^jp_b^+(t)}_{W^{3-j-1/p,p}(\gamma_+)} \right\} \ll 1.\nonumber
\end{align}
\end{theorem}
\begin{proof}
We solve \eqref{eulerv} and \eqref{BC2} via the velocity (transport) equations, and the Poisson equation \eqref{ppaeq}  with $F=F(v)$ given by \eqref{ppf} and \eqref{BC5} with $g_b^-$ given by \eqref{pbbc1}.

{\it A Priori Estimates.} We assume  $\sum_{j=0}^{1}\|\p_t^j v\|_{W^{2-j,p}(\Omega)}\le \varepsilon\ll1$. Since $\nabla  g_b^-\cdot\nu=0$ on $\gamma_- \cap \gamma_0$, $\nabla p_b^+\cdot\nu=0$ and $\nabla \p_t p_b^+\cdot\nu=0$ on $\gamma_+ \cap \gamma_0$, by part (2) of Theorem \ref{vomthp1}:
 \begin{align}\label{2vn}
 \sum_{j=0}^1\norm{
\p_t^j p
}_{W^{3-j,p}(\Omega)}
&\ls   \left(\norm{\nabla U}_{W^{1,p}(\Omega)}+\varepsilon\right)\sum_{j=0}^1\norm{
 \p_t^j   v
}_{W^{2-j,p}(\Omega)}
\nonumber
\\&\quad+\sum_{j=0}^1  \left\{\norm{\partial_t^j g_b^- }_{W^{2-j-1/p,p}(\gamma_-)}+      \norm{\partial_t^j p_b^+ }_{W^{3-j-1/p,p}(\gamma_+)} \right\} .
\end{align}
We deduce from the condition \eqref{thcom11} on $\gamma_-$ and \eqref{thcom22} the compatibility conditions
\beq
\p_1 v_{0,i}=- \frac{1}{U+v_{b,1}^-(0)}(\partial_t v_{b,i}^-(0) + \sum_{j=2,3}v_{b,i}^-(0)\p_j v_{b,i}^-(0)+  \p_3  p_0) \text{ on }\gamma_-,\ i=2,3,
\eeq
then by treating $\nabla p$ as a forcing term in the momentum equations, upon applying one spatial derivative, we follow the same proof as in Theorem \ref{scalar2} to obtain
\begin{align}\label{2omegaesn}
&\sup_{0\le t<\infty}\sum_{j=0}^1\norm{
 \p_t^j v (t)
}_{W^{2-j,p}(\Omega)}
\ls  \norm{
 v_0
}_{W^{2,p}(\Omega)}+     \sum_{j=0}^2\norm{
 \p_t^j v_b^-
}_{L^\infty(0,\infty;W^{2-j,p}(\gamma_-))}
\\&  \quad+ {c_1^{-1}}
\sup_{0\le t<\infty}\sum_{j=0}^1\norm{
\p_t^j   p (t)
}_{W^{3-j,p}(\Omega)}  +  \left({c_1^{-1}} \norm{\nabla U}_{W^{1,p}(\Omega)} +\varepsilon\right) \sup_{0\le t<\infty}\sum_{j=0}^1  \norm{
\p_t^j  v (t)
}_{W^{2-j,p}(\Omega)}    .\nonumber
\end{align}
We then conclude \eqref{2goodes} from \eqref{2vn} and \eqref{2omegaesn} by $
\norm{\nabla U}_{W^{2,p}(\Omega)}< \delta  c_1
$.

{\it Skecth of Iterations to Justify \eqref{2vn} and \eqref{2omegaesn}.} We start with $v^{(1)}(0)=v_0$.  Suppose $v^{(n)}$ is given with  $v^{(n)}(0)=v_0$ and  $v^{(n)}\cdot \nu=0$ on $\gamma_0$ (and  $v^{(n)}_\theta=0$ on $\gamma_0$  for the axisymmetric circular pipe flow), we construct $p^{(n)}$ as the solution to the Poisson equation
\begin{equation} \label{2last1}
- \Delta  p^{(n)}  =  F(v^{(n)}) \text{ in }\Omega,\
 \nabla p^{(n)}\cdot \nu=0 \text{ on }\gamma_0,\  \p_1    p^{(n)}=  g_b^-\text{ on }{  \gamma}_-, \   p^{(n)}=  p_b^+\text{ on }{  \gamma}_+ .
\end{equation}
With such $p^{(n)}$, we then construct $v^{(n+1)}$ as the solution to the transport equations
\begin{equation} \label{2last2}
\begin{cases}
\p_tv^{(n+1)}+(u_s+v^{(n)})\cdot\nabla v^{(n+1)} =- v^{(n)} \cdot\nabla   u_s+  \nabla p^{(n+1)} \quad\text{in }\Omega,\\
v^{(n+1)}|_{t=0}=v_0,\quad v^{(n+1)}|_{\gamma_-}=v_b^-.
\end{cases}
\end{equation}
Note that $v^{(n)}(0)=v_0$ and so valuing \eqref{2last1} at $t=0$ and \eqref{omegaeq0} lead to $p^{(n)}(0)=p_0$, and hence $v_0$ and $v_b^-$ satisfy the necessary compatibility conditions for the $W^{2,p}$ solvability of \eqref{2last2} by the condition \eqref{thcom11} on $\gamma_-$ and \eqref{thcom22} since $v^{(n)}\cdot \nu=0$ on $\gamma_0$.
Moreover, by projecting the equations along $\nu$ onto $\gamma_0$, we find that for the square pipe flow,
\beq
\p_t (v^{(n+1)}\cdot\nu)+ (u_s+v^{(n)} )\cdot \nabla (v^{(n+1)}\cdot\nu)  =0\text{ on }\gamma_0.
\eeq
Since  $v_0\cdot\nu=0$ on $\gamma_0$ and $v_b^-\cdot\nu=0$ on $\gamma_- \cap \gamma_0$, we then have
$
v^{(n+1)}\cdot\nu=0\text{ on }\gamma_0. 
$
While for the axisymmetric circular pipe flow, since $v_\theta^{(n)}=0$ on $\gamma_0$,
\beq\label{aaa}
\partial_t v_\theta^{(n+1)}+ (U+v^{(n)}_1)\partial_1 v_\theta^{(n+1)} = 0.
\eeq
By $v_{0,\theta}=0$ on $\gamma_0$ and $v_{b,\theta}^-=0$ on $\gamma_-\cap \gamma_0$, we then have
$
v^{(n+1)}_\theta=0
$ on $\gamma_0$.

Now we pass $n\rightarrow\infty$ to have the limit $(v,p)$ satisfying \eqref{eulerv} and \eqref{BC2} except ${\rm div} v=0$ in $\Omega$. It then follows by applying ${\rm div}$ to the velocity equations and the Poisson equation that
\beq
\partial_t {\rm div} v+(u_s+v)\cdot\nabla {\rm div} v     =-\p_i(u_s+v)_j \p_j(u_s+v)_i-\Delta p=0\quad \text{in }\Omega,
\eeq
and by solving from the first component of the velocity equations and using the BCs that
\beq
{\rm div} v=\frac{1}{(U+v_{b,1}^-)} (\partial_t v_{b,1}^- +\sum_{i=2,3}v_{b,i}^-(\p_i v_{b,1}^-+\p_i U)+\p_1    p )+ \sum_{i=2,3} \p_i v^-_{b,i}=0\text{ on }\gamma_-.
\eeq
Since ${\rm div} v_{0} =0$ in $\Omega$, we then have $
{\rm div} v=0$ in $\Omega.$   \end{proof}

\subsection{Velocity-velocity BC Problem for Narrow Pipes}
The main result of this section is

\begin{theorem}\label{th22}
Let $3<p<\infty$. Assume \eqref{bbcondition},  $
v_{b}^-\cdot\nu=0$ and $\nabla  g_b^-\cdot\nu=0$ on $\gamma_- \cap \gamma_0$,
and  the compatibility conditions:
\begin{align}\label{3thcom22}
&\dive v_0=0\text{ in }\Omega,\ v_0\cdot \nu=0\text{ on }\gamma_0,\ v_{0}=v_b^-(0) \text{ on }\gamma_-, \ v_{0,1}=v_{b,1}^+(0)\text{ on }\gamma_+,
\\
\label{3thcom2200}
&-(u_s+v_0)\cdot\nabla v_{0,i} -  \p_i  p_0=\p_tv_{b,i}^-(0)\text{ on }\gamma_-, \ i=2,3,
\end{align}
where $p_0$ is the unique solution with $\int_\Omega p_0=0$ to
\begin{align}\label{omegaeq022}
- \Delta  p_0  =  F(v_0) \ \text{in }\Omega,
\ \nabla p_0\cdot \nu=0 \text{ on }\gamma_0,\   \p_1    p_0=  g_b^-(0)\text{ on }{  \gamma}_-, \   \p_1    p_0=g_b^+(v_0)\text{ on }{  \gamma}_+.
\end{align}
$(i)$ for $D=(-h,h)^2$, assume further \eqref{oscondition31}  and  \eqref{oscondition310};

\hspace{-12pt}$(ii)$ for the axisymmetric circular pipe flow with $D=B(0,h)$,  assume further that $v_{0,\theta}=0\text{ on }\gamma_0$, 

\ $
v_{b,\theta}^- =0\text{ on }\gamma_-\cap \gamma_0$, \eqref{oscondition32} and \eqref{oscondition320}.

\noindent There exist constants  $\delta_1>0$  and $\delta_h>0$ such that if
$
0<h\le \delta_1$ and $\norm{\nabla U}_{W^{2,p}(\Omega)}< \delta_h \min_{\bar\Omega} U  ,
$
then there is a unique solution to \eqref{eulerv}  and \eqref{BC3} on $[0,\infty)$ satisfying
\begin{align}\label{goodes22}
&\sup_{0\le t<\infty} \sum_{j=0}^{1} \norm{\p_t^j v(t)}_{W^{2-j,p}(\Omega)}
\\&  \lesssim \norm{v_0}_{W^{2,p}(\Omega)}+\!\sup_{0\le t<\infty} \!\left\{\!\sum_{j=0}^2 \norm{\partial_t^jv_b^-(t)}_{W^{3-j-1/p,p}(\gamma_-)}+ \sum_{j=0}^1 \norm{\partial_t^j v_{b,1}^+(t)}_{W^{2-j-1/p,p}(\gamma_+)}\! \right\}:=\varepsilon_0\ll1.\nonumber
\end{align}
\end{theorem}

Here the $\omega_b^-$ appearing in \eqref{oscondition310} (resp. \eqref{oscondition320}) is defined on $\gamma_-$ by \eqref{kap} below, and its value depends on $p$ in general, however, the value of $\omega_b^-\times \nu$ (resp. $\omega_b^-$) on $\gamma_-\cap\gamma_0$ for the square pipe flow (resp. the circular pipe flow) does not depend on $p$ due to the BC $\nabla p\cdot \nu=0$ on $\gamma_0$.
\begin{proof}
We solve \eqref{eulerv} and \eqref{BC3} via the vorticity equations \eqref{omegaeq} with $\omega_{b}^-=\omega_{b}^-(p)$ given by
\begin{equation}\label{kap}
\begin{cases}
\dis\omega_{b,1}^-:=\partial_2 v_{b,3}^--\partial_3 v_{b,2}^-,\\
\dis\omega_{b,2}^-: =   \partial_3 v_{b,1}^- + \frac{1}{U+v_{b,1}^-}(\partial_t v_{b,3}^- + \sum_{i=2,3}v_{b,i}^-\p_i v_{b,3}^-+  \p_3  p),
\\\dis\omega_{b,3}^- :=   -\partial_2 v_{b,1}^- - \frac{1}{U+v_{b,1}^-} (\partial_t v_{b,2}^- + \sum_{i=2,3}v_{b,i}^-\p_i v_{b,2}^- +  \p_2  p),
\end{cases}
\end{equation}
the $div$--$curl$ system \eqref{omegaeqv}, and the Poisson equation \eqref{ppaeq} with $F=F(v)$ given by \eqref{ppf} and \eqref{BC5} with $g_b^-$ given by \eqref{pbbc1} and $g_b^+=g_b^+(v)$ given by \eqref{pbbc2}.

{\it A Priori Estimates.} We assume  $|\|(v,p)\||:=\sum_{j=0}^{1}(\|\p_t^j v\|_{W^{2-j,p}(\Omega)}+\|\p_t^j p \|_{W^{2-j,p}(\Omega)})\le \varepsilon\ll1$.

{\it Step 1. Reduction to $H^1$ Stability.} We deduce from the condition \eqref{3thcom22} on $\gamma_-$ and \eqref{3thcom2200} the compatibility condition $\omega_0=\omega_{b}^-(0)$ on $\gamma_-$ for \eqref{omegaeq}, then by part (1) of Theorem \ref{vorth}:
  \begin{align}\label{3vnvv1}
&\sup_{0\le t<\infty}\sum_{j=0}^1\norm{
 \p_t^j \omega(t)
}_{W^{1-j,p}(\Omega)}^p
\\&\ls_h  \norm{
 \omega_0
}_{W^{1,p}(\Omega)}^p+     {c_2}{c_1^{-1}} \sum_{j=0}^1\norm{
 \p_t^j \omega_b^-(p)
}_{L^\infty(0,\infty;W^{1-j,p}(\gamma_-))}^p +   ( \delta_h+\varepsilon)\sup_{0\le t<\infty}\sum_{j=0}^1\norm{
\p_t^j  v(t)
}_{W^{2-j,p}(\Omega)}^p.\nonumber
\end{align}
The assumptions of Theorem \ref{vomth} will be satisfied so that
  \begin{align}\label{3vnvv2}
&\sum_{j=0}^1\norm{
\p_t^j  v
}_{W^{2-j,p}(\Omega)}
\ls_h  \sum_{j=0}^1\norm{
 \p_t^j \omega
}_{W^{1-j,p}(\Omega)}+ \sum_{j=0}^1  \norm{\partial_t^j v_{b,1}^\pm}_{W^{2-j-1/p,p}(\gamma_\pm)}.
\end{align}
Since $\nabla  g_b^-\cdot\nu=0$ on $\gamma_- \cap \gamma_0$, by part (2) of Theorem \ref{vomthp2}:
 \begin{align}\label{3vnvv3}
 \sum_{j=0}^1\norm{
\p_t^j p
}_{W^{3-j,p}(\Omega_-)}\nonumber
 &\ls_h  \left(\norm{\nabla U}_{W^{1,p}(\Omega)} +\varepsilon\right)\sum_{j=0}^1\norm{
 \p_t^j   v
}_{W^{2-j,p}(\Omega)}
\nonumber
\\&\quad+ \sum_{j=0}^1  \norm{\partial_t^j g_b^-}_{W^{2-j-1/p,p}(\gamma_-)}+\sum_{j=0}^1  \norm{\partial_t^j g_b^+(v)}_{L^{2}(\gamma_+)}.
\end{align}
By  $W^{3-j,p}(\Omega_-) \subset\subset W^{2-j,p}(\gamma_-)$ and the interpolation, replacing $\p_t v$ in the last term of \eqref{3vnvv3} by equations and using part (1) of Theorem \ref{vomthp2}, we then deduce from \eqref{3vnvv1}--\eqref{3vnvv3} that
\begin{align}  \label{333vn}
 \sup_{0\le t<\infty} |\!|\! |(v,p)(t) |\!|\!|
  \lesssim_h \varepsilon_0+ \sup_{0\le t<\infty} \norm{
v(t)
}_{L^{2}(\Omega)} .
\end{align}

{\it Step 2. Standalone $H^1$ Stability for $h\ll1$.} Note that there is no smallness in closing \eqref{333vn}. We now establish $H^1$ stability by estimates with explicit $h$ dependence and claim that for $h\ll1,$
\beq\label{1213}
\sup_{0\le t<\infty} \norm{
v(t)
}_{H^{1}(\Omega)}^2\ls_h \varepsilon_0^2+\varepsilon_0\sup_{0\le t<\infty} |\!|\! |(v,p)(t) |\!|\!| + \varepsilon\sup_{0\le t<\infty}|\!|\! |(p,v)(t) |\!|\!|^2.
\eeq
We then conclude \eqref{goodes22} from \eqref{333vn} and \eqref{1213}.

We may  focus only on axisymmetric circular pipe flows and square pipe flows follow similarly.

$\bullet$ \underline{Estimating $v$ in terms of $\omega$.} Note that
$
  |\omega |^2 = |\p_1 v_\theta |^2+|r^{-1}\p_r(r v_\theta)|^2+|\p_1 v_r|^2 +|\p_r v_1|^2-2\p_1 v_r  \p_r v_1
$, and we may integrate by parts in $r$ and $x_1$ to have, since $v_r=0$ on $\gamma_0$,
\[
\int_\Omega  \p_1 v_r  \p_r v_1
= -\int_\Omega  \frac{1}{r}\p_r\p_1 (r v_r ) v_1  = \int_\Omega  \frac{1}{r}\p_r(r v_r ) \p_1 v_1  -\int_{\gamma_+}\frac{1}{r}\p_r  (r v_r ) v_1v_{1,b}^++\int_{\gamma_-}\frac{1}{r}\p_r  (r v_r ) v_{1,b}^-.
\]
These imply, since ${\rm div} v=\frac{1}{r}\p_r  (r v_r)
+\p_1 v_1 =0$,
\beq\label{claiml1}
\norm{\p_1 v}_{L^2(\Omega)}^2+\norm{\p_r v_1}_{L^2(\Omega)}^2
+\|r^{-1}\p_r(r v_r)\|_{L^2(\Omega)}^2+\|r^{-1}\p_r(r v_\theta)\|_{L^2(\Omega)}^2 \le \norm{\omega}_{L^2(\Omega)}^2+C_h  \varepsilon_0  |\!|\! |v |\!|\!|.
\eeq

$\bullet$ \underline{Estimating $\omega$ in terms of $p$.} The transport vorticity estimates and the inflow BC of $\omega$ yield
\begin{align}\label{claiml2}
& \sup_{0\le t<\infty}\norm{ \omega  (t)}_{L^2(\Omega)}^2
 \le  c_2 c_1 ^{-1} \!\!   \norm{ \omega_b^-(p) }_{L^\infty(0,\infty;L^{2}(\gamma_-))}^2+C_h (\varepsilon_0^2+\!\sup_{0\le t<\infty} \! (\varepsilon_0|\!|\! |(v,p)(t) |\!|\!|+\varepsilon |\!|\! |(v,p)(t) |\!|\!|^2))\nonumber
\\&\quad\le  c_2 c_1^{-3}    \norm{\p_r p  }_{L^\infty(0,\infty;L^{2}(\gamma_-))}^2+C_h (\varepsilon_0^2+\sup_{0\le t<\infty} (\varepsilon_0|\!|\! |(v,p)(t) |\!|\!|+ \varepsilon|\!|\! |(v,p)(t) |\!|\!|^2)).
\end{align}

$\bullet$ \underline{Estimating $p$ in terms of $v$.}
Applying the Poincar\'e type argument, we have
\begin{align}\label{lastp1}
 \norm{ \p_r p }_{L^{2}(\gamma_-)}^2
 &\ls \norm{ \p_r p }_{L^{2}(\Omega_-)}^2+ \norm{ \p_r p }_{L^{2}(\Omega_-)}\norm{\p_1\p_r p }_{L^{2}(\Omega_-)} \nonumber
\\ &\ls h^2  \norm{r^{-1}\p_r (r\p_r p )  }_{L^{2}(\Omega_-)}^2+h   \norm{r^{-1}\p_r (r\p_r p )  }_{L^{2}(\Omega_-)} \norm{\p_1\p_r p }_{L^{2}(\Omega_-)}   .
\end{align}
Note that $|F|^2=|\Delta  p |^2 = |\p_{11}p|^2+ |r^{-1}\p_r (r\p_r p )|^2
  +2r^{-1}\p_r (r\p_r p )
\p_{11}p $, and we integrate by parts in $r$ and $x_1$ to have, since $\p_r p=0$ on $\gamma_0$,
\[\int_\Omega  \frac{1}{r}\p_r(r\p_r p)
\p_{11}p\chi^2 =-\int_\Omega   \p_r p
\p_r\p_{11}p\chi^2=\int_\Omega   |\p_1 \p_r p|^2\chi^2+2\int_\Omega   \p_r p
\p_r\p_{1}p\chi \p_1\chi+\int_{\gamma_-}   \p_r p  \p_r
g_b^-,\]
where $\chi=\chi(x_1) $ is a smooth function with $\chi=1$ for $x_1\le 0$ and $\chi=0$ for $x_1\ge 1/2$, and hence,
\begin{align}\label{lastp2}
\norm{\p_{11}p}_{L^{2}(\Omega_-)}^2+\|r^{-1}\p_r (r\p_r p ) \|_{L^{2}(\Omega_-)}^2+\norm{\p_1 \p_r p }_{L^{2}(\Omega_-)}^2
 \ls \norm{ \p_{1}p}_{L^2(\Omega)}^2+C_h\varepsilon_0  |\!|\! |p |\!|\!|.
\end{align}
On the other hand, integrating \eqref{ppaeq} multiplied by $p$ by parts over $\Omega$ and by \eqref{BC5}, we have
\[
\int_\Omega |\nabla p|^2 =- \int_\Omega  F p +\int_{\gamma_+}\p_1 p p -\int_{\gamma_-}\p_1 p p
\le \int_{\gamma_+}U\frac{1}{r}\p_r  (r v_r)  p  +C_h(\varepsilon_0^2+  \varepsilon_0  |\!|\! |p |\!|\!|+ \varepsilon|\!|\! |(p,v) |\!|\!|^2).
\]
Using further the Newton--Leibniz formula in $x_1\in (-h,0)$ and integrating by parts in $r$ and $x_1$, since $v=v_b$ on $\gamma_-$ and $v_r=0$ on $\gamma_0$, we deduce
\begin{align}
 \int_{\gamma_+}U\frac{1}{r}\p_r  (r v_r)  p &= \int_\Omega   \frac{1}{r}\p_1 \p_r  (r v_r) U p +\int_\Omega\frac{1}{r}\p_r  (r v_r) \p_1(U  p ) +\int_{\gamma_-}U\frac{1}{r}\p_r  (r v_{b,r}^-)  p\nonumber
\\&
\le-\int_\Omega    \p_1   v_r U \p_r  p +\int_\Omega\frac{1}{r}\p_r  (r v_r)U\p_1 p +C_h( \varepsilon_0  |\!|\! |p |\!|\!|+ \varepsilon|\!|\! |(p,v) |\!|\!|^2).
\nonumber
\end{align}
These two imply
\begin{align}\label{lastp3}
\norm{\nabla p}_{L^2(\Omega)}^2 \le
c_2^2 \norm{\p_1 v_r}_{L^2(\Omega)}^2+c_2^2
\|r^{-1}\p_r  (r v_r)\|_{L^2(\Omega)}^2  +C_h(\varepsilon_0^2+  \varepsilon_0  |\!|\! |p |\!|\!|+ \varepsilon|\!|\! |(p,v) |\!|\!|^2).
\end{align}
Combining \eqref{lastp1}--\eqref{lastp3} yields
\begin{align}\label{claiml3}
 \norm{ \p_r p }_{L^{2}(\gamma_-)}^2 \ls h c_2^2 (\norm{\p_1 v_r}_{L^2(\Omega)}^2+ 
\|r^{-1}\p_r  (r v_r)\|_{L^2(\Omega)}^2) +C_h(\varepsilon_0^2+  \varepsilon_0  |\!|\! |p |\!|\!|+ \varepsilon|\!|\! |(p,v) |\!|\!|^2) .
\end{align}

Consequently, chaining \eqref{claiml1}, \eqref{claiml2} and \eqref{claiml3} concludes the claim \eqref{1213} for $h\ll 1$.

{\it Sketch of Iterations to Justify \eqref{333vn} and \eqref{1213}.} We start with  $(v^{(1)},p^{(1)})=(v_0,p_0)$. Suppose $(v^{(n)},p^{(n)})$ are given with  $(v^{(n)},p^{(n)})(0)=(v_0,p_0)$, $\dive v^{(n)}=0$ in $\Omega$, and $v^{(n)}\cdot \nu=0$ and $\nabla p^{(n)}\cdot \nu=0$ on $\gamma_0$ (and $v^{(n)}_\theta =0$  on $\gamma_0$ for the axisymmetric circular pipe flow). We construct $\omega^{(n+1)}$ as the solution to the vorticity equations
\begin{equation} \label{3last2vor}
 \begin{cases}\!
\p_t\omega^{(n+1)}\!+\!(u_s+v^{(n)})\!\cdot\!\nabla \omega^{(n+1)}\! =\!- v^{(n)} \!\cdot\!\nabla   \omega_s\!+  \omega_s \!\cdot\! \nabla  v^{(n)}\!+  \omega^{(n+1)} \!\cdot\! \nabla \!( u_s +    v^{(n)})\! \text{ in }\Omega,\\\!
\omega^{(n+1)}|_{t=0}=\omega_0:=\curl v_0,
\end{cases}
\end{equation}
with the inflow BC of $\omega^{(n+1)}=\omega_b^-(v^{(n)},p^{(n)})$ on $\gamma_-$ given by
\begin{align}\label{3lastbd1}
\begin{cases}
\omega_{b,1}^{-}:=\partial_2 v_{b,3}^--\partial_3 v_{b,2}^-, \\
\dis\omega_{b,2}^{-}(v^{(n)},p^{(n)}):=-\p_3U+      \frac{1}{U+v_1}\left\{\partial_t v_{b,3}^-+  \p_3 \left(\frac{1}{2} |u_s+v_b^-|^2+    p^{(n)}\right)+  v_2^{(n)}\omega_{b,1}^{-}\right\}, \\
\dis\omega_{b,3}^{-}(v^{(n)},p^{(n)}): =  \partial_2 U-\frac{1}{U+v_1}\left\{\partial_t v_{b,2}^- +  \p_2 \left(\frac{1}{2} |u_s+v_b^-|^2+    p^{(n)}\right)- v_3^{(n)}\omega_{b,1}^{-}\right\}.
\end{cases}
\end{align}
Note that $(v^{(n)},p^{(n)})(0)=(v_0,p_0)$ guarantees the compatibility condition $
\omega_0=\omega_b^-(v^{(n)},p^{(n)})(0)$ on $\gamma_-$ from \eqref{3thcom22} and \eqref{3thcom2200}.   Here, following the idea of \cite{AKM}, the BCs in \eqref{3lastbd1} are designed so that  $ \dive \omega^{(n+1)}=0$ in $\gamma_-$ (see Lemma 3.1 in \cite{AKM} for the details), which implies further $ \dive \omega^{(n+1)}=0$ in $\Omega$ as in part (2) of Theorem \ref{vorth}.  On the other hand, for $D=(-h,h)^2$ note  from \eqref{oscondition310} that $ \omega_{b,1}^-=0$ on $\gamma_-\cap\gamma_0$ and  so $\omega_{b}^{-}\times \nu\equiv\omega_{b}^{-}(0,0)\times \nu=0$ on $\gamma_- \cap \gamma_0$ since $\nabla p^{(n)}\cdot \nu=0$ on $\gamma_0$,   which together with \eqref{oscondition310} implies $ \omega^{(n+1)}\times \nu=0$ on $ \gamma_0$ as in part (3) of Theorem \ref{vorth}; similarly, $ \omega^{(n+1)}=0$ on $\gamma_0$ for $D=B(0,h)$.

Now with such $\omega^{(n+1)}$ we then construct $v^{(n+1)}$ as the solution to the {\it div}--{\it curl} problem
\begin{equation} \label{3last122222}
\begin{cases}
{\rm curl\,}v^{(n+1)}=\omega^{(n+1)},\  \dive v^{(n+1)}=0\ \,\text{in }\Omega,
\\ v^{(n+1)}\cdot\nu= 0 \  \text{on }\gamma_0 ,\ \, v_1^{(n+1)}=v_{b,1}^\pm   \ \text{on }\gamma_\pm.
\end{cases}
\end{equation}
We now show $v^{(n+1)}_\theta=0$ on $\gamma_0$  for the axisymmetric circular pipe flow. First, since $\dive v^{(n)}=0$ in $\Omega$, we deduce from \eqref{3last2vor} and \eqref{3last122222} that there exists an axisymmetric $\pi^{(n+1)}$ so that
\begin{align}\label{3llhh1}
\partial_t  v^{(n+1)} =& \nabla v^{(n+1)}\cdot u_s
+\nabla v^{(n+1)}\cdot v^{(n)}
+\nabla u_s\cdot v^{(n)}
\nonumber\\&-(u_s\cdot\nabla)v^{(n+1)}
-(v^{(n)}\cdot\nabla)v^{(n+1)}
-(v^{(n)}\cdot\nabla)u_s -\nabla\pi^{(n+1)} .
\end{align}
Since $v^{(n)}_\theta =0$ and $v^{(n)}_r =0$ on $\gamma_0$, it follows from \eqref{3llhh1} that the crucial \eqref{aaa} holds also here.
Note that $\omega^{(n+1)}=0$ on $\gamma_0$ implies in particular $\p_1  v_\theta^{(n+1)}  =0$ on $\gamma_0$, and hence \eqref{aaa} reduces to be $\partial_t  v_\theta^{(n+1)} =0$ on  $\gamma_0$; on the other hand, as in the step 2 of the proof of Theorem \ref{th2} we have
  $ v^{(n+1)}(0)=v_0$ and so $ v_\theta^{(n+1)}(0) =0$ on $\gamma_0$ since $v_{0,\theta}=0$ on $\gamma_0$, and then
$v_\theta^{(n+1)}=0$ on $\gamma_0$.

Now, we construct $p^{(n+1)}$ as the solution to the Poisson equation
\begin{equation} \label{3last1vv}
\begin{cases}
\dis- \Delta  p^{(n+1)}  =  F(v^{(n+1)})-\Phi^{(n+1)},\ \Phi^{(n+1)}:=\frac{2}{|\Omega|} \int_{\gamma_-}\sum_{i=2,3} \p_i(U+v^{-}_{b,1})  (v_i^{(n+1)} -  v_{b,i}^-) \quad\text{in }\Omega,
\\ \nabla p^{(n+1)}\cdot \nu=0 \text{ on }\gamma_0,\ \p_1    p^{(n+1)}=  g_b^-\text{ on }\gamma_-, \  \p_1    p^{(n+1)}=g_b^+(v^{(n+1)})\text{ on }\gamma_+ .
\end{cases}
\end{equation}
Here $\Phi^{(n+1)}$ is included so that the necessary solvability condition of \eqref{3last1vv} is satisfied, where $v_1^{(n+1)}=v_{b,1}^+$  on $\gamma_+$ and $\eqref{bbcondition}$ are used. It should be remarked that there is no derivative term of $v^{(n+1)} -  v_{b}^-$ in $\Phi^{(n+1)}$ via a crucial integration by parts over $\gamma_-$ observed in \cite{AKM} (see (2.25) on page 160 in \cite{AKM}). Note that $ v^{(n+1)}(0)=v_0$ implies $\Phi^{(n+1)}(0)=0$ and hence $ p^{(n+1)}(0)=p_0$.

Now we pass $n\rightarrow\infty$ to have the limits $(v,p)$ and a $\pi$ such that
\begin{align} 
&\partial_t v +(u_s+v)\cdot\nabla v + v \cdot\nabla  u_s +  \nabla  \pi = 0 \quad\text{in }\Omega,\\
&\dis- \Delta  p   =  F(v) -\Phi,\ \Phi:=\frac{2}{|\Omega|}\int_{\gamma_-}\sum_{i=2,3} \p_i(U+v^{-}_{b,1}) (v_i - v_{b,i}^-)  \quad\text{in }\Omega ,
\end{align}
together with ${\rm div} v=0$ in $\Omega$  and the BC \eqref{BC3} except $v_i=v_{b,i}^-$ on $\gamma_-$, $i=2,3$. 
We can then follow the ingenious proof of Lemma 2.1 in \cite{AKM} to prove that $\pi=p$ and $v_i=v_{b,i}^-$ on $\gamma_-$, $i=2,3$ (and hence $\Phi=0$) simultaneously.
\end{proof}

\section*{Acknowledgements}
Yan Guo thanks Zhuolun Yang for stimulating discussions of stable shear profiles for 2D channel flows.
Yan Guo is supported in part by NSF grant DMS-2405051. Yanjin Wang is supported by NSFC grants 12571251 and 12231016.


\vspace{0.5cm}
\begin{thebibliography}{99}



\bibitem{A}
Amadori, Debora.
Initial-boundary value problems for nonlinear systems of conservation laws.
\emph{NoDEA Nonlinear Differential Equations Appl.} \textbf{4} (1997), no. 1, 1--42.

\bibitem{ABDG}
Amrouche, C.; Bernardi, C.; Dauge, M.; Girault, V.
Vector potentials in three-dimensional non-smooth domains.
\emph{Math. Methods Appl. Sci.} \textbf{21} (1998), no. 9, 823--864.

\bibitem{AS}
Amrouche, Ch\'erif; Seloula, Nour El Houda.
$L^p$-theory for vector potentials and Sobolev's inequalities for vector fields: application to the Stokes equations with pressure boundary conditions.
\emph{Math. Models Methods Appl. Sci.} \textbf{23} (2013), no. 1, 37--92.

\bibitem{AKM}
Antontsev, S. N.; Kazhikhov, A. V.; Monakhov, V. N.
Boundary value problems in mechanics of nonhomogeneous fluids.
Stud. Math. Appl., 22 North-Holland Publishing Co., Amsterdam, 1990.

\bibitem{arnold}
Arnold, Vladimir I.
Conditions for nonlinear stability of stationary plane curvilinear flows of an ideal fluid.
\emph{Sov. Math. Dokl.} \textbf{6} (1965), 773--777.

\bibitem{BM}
Bedrossian, Jacob; Masmoudi, Nader.
Inviscid damping and the asymptotic stability of planar shear flows in the 2D Euler equations.
\emph{Publ. Math. Inst. Hautes \'Etudes Sci.} \textbf{122} (2015), 195--300.






\bibitem{BB}
Bianchini, Stefano; Bressan, Alberto.
Vanishing viscosity solutions of nonlinear hyperbolic systems.
\emph{Ann. of Math. (2)} \textbf{161} (2005), no. 1, 223--342.

\bibitem{BCP}
Bressan, Alberto; Crasta, Graziano; Piccoli, Benedetto.
Well-posedness of the Cauchy problem for $n\times n$ systems of conservation laws.
\emph{Mem. Amer. Math. Soc.} \textbf{146} (2000), no. 694, viii+134 pp.


\bibitem{BLY}
Bressan, Alberto; Liu, Tai-Ping; Yang, Tong.
$L^{1}$ stability estimates for $n\times n$ conservation laws.
\emph{Arch. Ration. Mech. Anal.} \textbf{149} (1999), no. 1, 1--22.


\bibitem{BDLIS15}
Buckmaster, Tristan; De Lellis, Camillo; Isett, Philip; Sz\'ekelyhidi, L\'aszl\'o, Jr.
Anomalous dissipation for 1/5-H\"older Euler flows.
\emph{Ann. of Math. (2)} \textbf{182} (2015), no. 1, 127--172.

\bibitem{BMNV}
Buckmaster, Tristan; Masmoudi, Nader; Novack, Matthew; Vicol, Vlad.
\emph{Intermittent Convex Integration for the 3D Euler Equations}.
Ann. of Math. Stud., 217. Princeton University Press, Princeton, NJ, 2023.

\bibitem{CH0}
Chen, Jiajie; Hou, Thomas Y.
Finite time blowup of 2D Boussinesq and 3D Euler equations with $C^{1,\alpha}$ velocity and boundary.
\emph{Comm. Math. Phys.} \textbf{383} (2021), no. 3, 1559--1667.

\bibitem{CH1}
Chen, Jiajie; Hou, Thomas Y.
Stable nearly self-similar blowup of the 2D Boussinesq and 3D Euler equations with smooth data I: Analysis.
(2022), arXiv:2210.07191.

\bibitem{CH}
Chen, Jiajie; Hou, Thomas Y.
Stable nearly self-similar blowup of the 2D Boussinesq and 3D Euler equations with smooth data II: Rigorous numerics.
\emph{Multiscale Model. Simul.} \textbf{23} (2025), no. 1, 25--130.

\bibitem{C}
Constantin, Peter.
On the Euler equations of incompressible fluids.
\emph{Bull. Amer. Math. Soc. (N.S.)} \textbf{44} (2007), no. 4, 603--621.

\bibitem{D}
Dafermos, Constantine M.
Polygonal approximations of solutions of the initial value problem for a conservation law.
\emph{J. Math. Anal. Appl.} \textbf{38} (1972), no. 1, 33--41.

\bibitem{Dbook}
Dafermos, Constantine M.
\textit{Hyperbolic conservation laws in continuum physics.}
Fourth edition. Grundlehren der mathematischen Wissenschaften [Fundamental Principles of Mathematical Sciences], 325. Springer-Verlag,
Berlin, 2016.

\bibitem{DS}
De Lellis, Camillo; Sz\'ekelyhidi, L\'aszl\'o, Jr.
The Euler equations as a differential inclusion.
\emph{Ann. of Math. (2)} \textbf{170} (2009), no. 3, 1417--1436.



\bibitem{DS13}
De Lellis, Camillo; Sz\'ekelyhidi, L\'aszl\'o, Jr.
Dissipative continuous Euler flows.
\emph{Invent. Math.} \textbf{193} (2013), no. 2, 377--407.


\bibitem{DM}
Donadello, Carlotta; Marson, Andrea.
Stability of front tracking solutions to the initial and boundary value problem for systems of conservation laws.
\emph{NoDEA Nonlinear Differential Equations Appl.} \textbf{14} (2007), no. 5-6, 569--592.

\bibitem{DE}
Drivas, Theodore D.; Elgindi, Tarek M.
Singularity formation in the incompressible Euler equation in finite and infinite time.
\emph{EMS Surv. Math. Sci.} \textbf{10} (2023), no. 1, 1--100.

\bibitem{E}
Elgindi, Tarek M.
Finite-time singularity formation for $C^{1,\alpha}$ solutions to the incompressible Euler equations on $\mathbb{R}^3$.
\emph{Ann. of Math. (2)} \textbf{194} (2021), no. 3, 647--727.


\bibitem{EGM}
Elgindi, Tarek M.; Ghoul, Tej-Eddine; Masmoudi, Nader.
On the stability of self-similar blow-up for $C^{1,\alpha}$ solutions to the incompressible Euler equations on $\mathbb{R}^3$.
\emph{Camb. J. Math.} \textbf{9} (2021), no. 4, 1035--1075.


\bibitem{FH}
Friedlander, Susan; Howard, Louis.
Instability in parallel flows revisited.
\emph{Stud. Appl. Math.} \textbf{101} (1998), no. 1, 1--21.



\bibitem{GKM}
Gie, Gung-Min; Kelliher, James P.; Mazzucato, Anna L.
The 3D Euler equations with inflow, outflow and vorticity boundary conditions.
\emph{J. Math. Pures Appl. (9)} \textbf{193} (2025), Paper No. 103628, 62 pp.

\bibitem{GKN}
Giri, Vikram; Kwon, Hyunju; Novack, Matthew.
The $L^3$-based strong Onsager theorem.
To appear in \emph{Ann. of Math.}, arXiv:2305.18509.


\bibitem{GR}
Giri, Vikram; Radu, R$\check{a}$zvan-Octavian.
The Onsager conjecture in 2D: a Newton-Nash iteration.
\emph{Invent. Math.} \textbf{238} (2024), 691--768.



\bibitem{G}
Glimm, James.
Solutions in the large for nonlinear hyperbolic systems of equations.
\emph{Comm. Pure Appl. Math.} \textbf{18} (1965), no. 6, 697--715.

\bibitem{Goodman}
Goodman, Jonathan Bernard.
\textit{Initial boundary value problems for hyperbolic systems of conservation laws.}
Thesis (Ph.D.)--Stanford University. 1983. 60 pp.

\bibitem{G1}
Grisvard, Pierre.
Elliptic Problems in Nonsmooth Domains.
Monographs and Studies in Mathematics, vol. 24. Pitman (Advanced Publishing Program), Boston, MA, 1985.


\bibitem{GKTT}
Guo, Yan; Kim, Chanwoo; Tonon, Daniela; Trescases, Ariane.
Regularity of the Boltzmann equation in convex domains.
\emph{Invent. Math.} \textbf{207} (2017), no. 1, 115--290.

\bibitem{GPW}
Guo, Yan; Pausader, Benoit; Widmayer, Klaus.
Global axisymmetric Euler flows with rotation.
\emph{Invent. Math.} \textbf{231} (2023), no. 1, 169--262.

\bibitem{GY}
Guo, Yan; Yang, Zhuolun.
Asymptotic stability of symmetric flows with viscous inflow boundary condition.
arXiv:2602.17059.

\bibitem{IJ}
Ionescu, Alexandru D.; Jia, Hao.
Non-linear inviscid damping near monotonic shear flows.
\emph{Acta Math.} \textbf{230} (2023), no. 2, 321--399.





\bibitem{I18}
Isett, Philip.
A proof of Onsager's conjecture.
\emph{Ann. of Math. (2)} \textbf{188} (2018), no. 3, 871--963.



\bibitem{J}
John, Fritz.
Formation of singularities in one-dimensional nonlinear wave propagation.
\emph{Comm. Pure Appl. Math.} \textbf{27} (1974), no. 4, 377--405.

\bibitem{KR1}
Kazhikhov, A. V.; Ragulin, V. V.
Nonstationary problems of the flow of an ideal fluid through a bounded region.
\emph{Dokl. Akad. Nauk SSSR} \textbf{250} (1980), no. 6, 1344--1347.

\bibitem{KR}
Kazhikhov, A. V.; Ragulin, V. V.
On a problem of flow for equations of an ideal fluid.
\emph{Zap. Nauchn. Sem. Leningrad. Otdel. Mat. Inst. Steklov. (LOMI)} \textbf{96} (1980), 84--96, 307.


\bibitem{K}
Kochin N. E.
On an existence theorem in hydrodynamics.
\emph{Prikl. Mat. Mekh.} \textbf{20} (1956), no. 2,  153--172.

\bibitem{Lin1}
Lin, Zhiwu. Instability of some ideal plane flows.
\emph{SIAM J. Math. Anal.} \textbf{35} (2003), no. 2, 318--356.

\bibitem{Lin2}
Lin, Zhiwu.
Nonlinear instability of ideal plane flows.
\emph{Int. Math. Res. Not.} (2004), no. 41, 2147--2178.

\bibitem{MZ}
Masmoudi, Nader; Zhao, Weiren.
Nonlinear inviscid damping for a class of monotone shear flows in a finite channel.
\emph{Ann. of Math. (2)} \textbf{199} (2024), no. 3, 1093--1175.


\bibitem{NV}
Novack, Matthew; Vicol, Vlad.
An intermittent Onsager theorem.
\emph{Invent. Math.} \textbf{233} (2023), no. 1, 223--323.



\bibitem{R}
Rayleigh, Lord.
On the Stability, or Instability, of certain Fluid Motions.
\emph{Proc. Lond. Math. Soc.} \textbf{11} (1879/80), 57--70.

\bibitem{RT}
Ren, Xiao; Tian, Gang.
Global solutions to the Euler--Coriolis system.
arXiv:2405.18390.


\bibitem{Sc}
Scheffer, Vladimir.
An inviscid flow with compact support in space-time.
\emph{J. Geom. Anal.} \textbf{3} (1993), no. 4, 343--401.

\bibitem{Sh}
Shnirelman, A.
On the nonuniqueness of weak solution of the Euler equation.
\emph{Comm. Pure Appl. Math.} \textbf{50} (1997), no. 12, 1261--1286.

\bibitem{WLGB}
Wang, Y.; Lai, C.-Y.; G\'omez-Serrano, J.; Buckmaster, T.
Asymptotic self-similar blow-up profile for three-dimensional axisymmetric Euler equations using neural networks.
\emph{Phys. Rev. Lett.} \textbf{130} (2023), no. 24, Paper No. 244002, 6 pp.


\bibitem{Y}
Yudovich, V. I.
A two-dimensional non-stationary problem on the flow of an ideal incompressible fluid through a given region.
\emph{Mat. Sb. (N.S.)} \textbf{64} (106) (1964), 562--588.

\bibitem{Z1}
Zajaczkowski, V. M.
Solvability in the small of a nonstationary flow problem for an ideal incompressible fluid. I.
\emph{Zap. Nauchn. Sem. Leningrad. Otdel. Mat. Inst. Steklov. (LOMI)} \textbf{96} (1980), 39--56, 306.

\bibitem{Z2}
Zajaczkowski, W. M.
Local solvability of nonstationary leakage problem for ideal incompressible fluid. II.
\emph{Pacific J. Math.} \textbf{113} (1984), no. 1, 229--255.
\end{thebibliography}
\end{document}